\documentclass[reqno,a4paper,11pt]{amsart}

\usepackage[english]{babel}

\usepackage{hyperref}

\usepackage{amsfonts}
\usepackage{amsmath}
\usepackage{amsthm}
\usepackage{amssymb}
\usepackage{indentfirst}
\usepackage{xcolor}
\usepackage{tabularx}
\usepackage{graphicx}
\usepackage{esint}
\allowdisplaybreaks

\usepackage{soul}

\usepackage[a4paper,top=3.5cm,bottom=1cm,left=2.5cm,right=2.5cm]{geometry}

\usepackage{subcaption}

\usepackage{mathrsfs}
\usepackage{mathtools}
\usepackage{enumitem}
\usepackage[normalem]{ulem}

\newcommand{\compresslist}{
	\setlength{\itemsep}{1.5pt}
	\setlength{\parskip}{0pt}
	\setlength{\parsep}{0pt}
}

\usepackage[utf8]{inputenc}
\usepackage[T1]{fontenc}
\usepackage[english]{babel}

\numberwithin{equation}{section}

\theoremstyle{plain}
\begingroup
\theoremstyle{plain}
\newtheorem{theorem}{Theorem}[section]

\newtheorem{proposition}[theorem]{Proposition}
\newtheorem{lemma}[theorem]{Lemma}
\theoremstyle{definition}

\theoremstyle{remark}
\newtheorem{remark}[theorem]{Remark}
\endgroup

\theoremstyle{definition}
\theoremstyle{remark}

\mathchardef\emptyset="001F

\newcommand{\R}{\mathbb{R}}

\newcommand{\B}{\mathrm{B}}
\newcommand{\dive}{\mathrm{div}}

\newcommand{\spt}{\mathrm{spt}}

\newcommand{\id}{\mathit{id}}

\newcommand{\T}{\mathcal{T}}
\newcommand{\di}{\mathrm{d}}

\newcommand{\F}{\mathcal{F}}

\newcommand{\J}{\mathcal{J}}

\newcommand{\Lip}{\mathrm{Lip}}
\newcommand{\argmin}{\mathrm{argmin}}
\newcommand{\spann}{\mathrm{span}}
\newcommand{\HS}{\mathrm{HS}}
\newcommand{\He}{\mathrm{H}}
\newcommand{\Q}{\mathrm{Q}}
\newcommand{\Qq}{\mathcal{Q}}
\newcommand{\dist}{\mathrm{dist}}

\newcommand{\sfd}{\mathsf{d}}

\newcommand{\cP}{\mathcal{P}}

\newcommand{\BL}{\mathrm{BL}}

\definecolor{dred}{rgb}{.8,0,0}
\definecolor{ddmagenta}{rgb}{0.7,0,0.9}
\definecolor{ddcyan}{rgb}{0,0.2,1.0}
\definecolor{Orchid}{rgb}{0.7,0.4,0}

\newcommand{\N}{\mathbb{N}}

\newcommand{\coloneq }{\hspace{1pt}\raisebox{0.74pt}{\scalebox{0.8}{:}}\hspace{-2.2pt}=}

\newcommand{\Pp}{\mathcal{P}}

\title[An alternate Lagrangian scheme for inhomogeneous games]{An alternate Lagrangian scheme for spatially inhomogeneous evolutionary games}
\author[S. Almi]{Stefano Almi}
\address[Stefano Almi]{Faculty of Mathematics, University of Vienna, 
Oskar-Morgenstern-Platz 1, 1090 Wien, Austria.}
\email{stefano.almi@univie.ac.at}

\author[M. Morandotti]{Marco Morandotti}
\address[Marco Morandotti]{Dipartimento di Scienze Matematiche ``G.~L.~Lagrange'',
Politecnico di Torino, Corso Duca degli Abruzzi, 24,
10129 Torino, Italy.}
\email{marco.morandotti@polito.it}

\author[F. Solombrino]{Francesco Solombrino}
\address[Francesco Solombrino]{Dipartimento di Matematica e Applicazioni ``R.~Caccioppoli'',
Universit\`a di Napoli Federico II, via Cintia, 80126 Napoli, Italy.}
\email{francesco.solombrino@unina.it}

\date{\today}


\begin{document}

\begin{abstract}
An alternate Lagrangian scheme at discrete times is proposed for the approximation of a nonlinear continuity equation arising as a mean-field limit of spatially inhomogeneous evolutionary games, describing the evolution of a system of spatially distributed agents with strategies, or labels, whose payoff depends also on the current position of the agents. The scheme is Lagrangian, as it traces the evolution of position and labels along characteristics and is alternate, as it consists of the following two steps: first the distribution of strategies or labels is updated according to a best performance criterion and then this is used by the agents to evolve their position. A general convergence result is provided in the space of probability measures. In the special cases of replicator-type systems and reversible Markov chains, variants of the scheme, where the explicit step in the evolution of the labels is replaced by an implicit one, are also considered and convergence results are provided.
\end{abstract}

\maketitle

\tableofcontents

\section{Introduction}
The capability of changing strategy as an adaptive response to the modification of the surrounding environment in order to maximize a certain payoff is of paramount importance in decision-making processes. 
Replicator-type models \cite{MR1635735} are a particular class of dynamical models that feature this adaptivity and are well suited for studying the evolution of strategies according to their success: given a pool of strategies, the occurrence of each of them evolves according to their performance with respect to all the others; in this way, if a strategy gives a payoff which is higher compared to the average of all strategies, it is enhanced, otherwise it is suppressed. 
This criterion, in the basic replicator model, is the only one that determines the evolution of the occurrence of the strategies, which in fact is independent from all other factors, in particular from the position of the agents that play those strategies. This is a reasonable assumption, not even a restrictive one, in many cases. For example, in a financial scenario, the set of (pure) strategies~$U$ contains the financial products available to an investor. Any combination of them, that is a portfolio, is called a mixed strategy: in a discrete setting such as this one, it corresponds to the fraction of the capital invested in each of the different financial products. Adapting the strategy means to allocate resources differently according to the evolution of the market, and the location the investor is at when making this decision is likely to not affect the reward of the portfolio. On the contrary, when the position influences the outcome, the system is more involved, as more feedback is available, and the adaptive optimization process relies on the mutual influence of position and strategy performance. We call such a system \emph{spatially inhomogeneous}, and make them the focus of this paper.

\smallskip

\paragraph{\textbf{Overview of the problem and state of the art.}}
The basic, spatially homogeneous, replicator equation of \cite{MR1635735} can be enriched to include spatial dependence of the payoff function: the idea is that the same strategy adopted in two different places might originate different rewards, precisely depending on the environment. Therefore, in order to maximize the payoff players can not only adapt their strategies, but also change their position seeking for the highest possible payoff. Spatially inhomogeneous evolutionary games, introduced in \cite{AmbForMorSav18}, provide a general mathematical framework for the evolution of a distribution of players with their (distributions of) strategies: a space-dependent replicator equation governs the evolution of the distribution~$\lambda\in\Pp(U)$ of the strategies~$u\in U$ while the evolution of the spatial variable~$x\in\R^{d}$ is determined by~$\lambda$.

In the subsequent contribution \cite{MorSol19}, this approach has been suitably extended as an abstract toolbox which is capable of rigorously describing the mean-field limit of a larger class of models which share the following features:
\begin{itemize}\compresslist
\item a multi-agent dynamics in which every agent is characterized by a label~$u\in U$ (accounting for different strategies or different populations to which each individual belongs);
\item exchange rates among the labels which are stochastic in nature and, therefore, are described by the evolution of a probability measure~$\lambda\in\Pp(U)$.
\end{itemize}
Several other models, besides the replicator dynamics mentioned above, are included in this class. 
The multi-label setting can be effectively used to describe situations in which the action of every individual is weighted differently according to the species it belongs to \cite{albi2014boltzmann,albi2017opinion,cirant2015multi,francesco2013measure,fornasier2014mean}. 
In the theory of mean-field games or in optimal control theory, labelling is used to distinguish informed agents in the evacuations of unknown environments, to highlight the influence of key investors in the stock market or of strong leaders in opinion formation \cite{bongini2016optimal,Burger,BMPW2009,Toscani2006}. 
The addition of source and sink terms in the spirit of \cite{Rossi-review} and of label switching \cite{thai} can be successfully dealt with in this class of models.
Relevant applications where label switching may occur come, for instance, from chemical reaction networks, where a particle may change its type as a result of the interaction with the others \cite{LLN2019,CRN,O1989}; also in social dynamics, loss or gain of opinion leadership over time is a natural postulate, as it happens in \cite[Section~3.b]{BMPW2009}.

The framework proposed in \cite{MorSol19} couples a nonlinear transport dynamics for the positions~$x\in\R^{d}$ of the agents with a Markov-type jump process for the labels~$\lambda\in\Pp(U)$ (see Section~\ref{s:hp}). The mean-field limit of the model was proved to be a nonlinear continuity equation of the form
\begin{equation}\label{NCE}
\partial_t\Psi_t+\dive(b_{\Psi_t}\Psi_t)=0 \,.
\end{equation}
in the space of probability measures over the pairs~$(x,\lambda)\in\R^d\times\Pp(U)$ driven by a velocity field~$b_{\Psi}(x,\lambda)$ depending on the global state of the system~$\Psi\in\Pp(\R^d\times\Pp(U))$.
These equations 
are part of a general class which is of great interest in the mathematical community \cite[Chapter~8]{AmbGigSav08} and can be studied both with a Lagrangian or a Eulerian approach.
On the one hand, the nonlinear continuity equation expresses the Eulerian point of view tracing the evolution of the global state~$\Psi$.
On the other hand, a notion of solution can also be provided by the Lagrangian point of view tracing the characteristics, which are, in our case, solutions to an ODE in a suitably constructed Banach space.

Given an initial datum~$\widehat\Psi$, a solution~$t\mapsto\Psi_t$ of the initial value problem for the nonlinear continuity equation is called a Eulerian solution, whereas a curve~$t\mapsto\Psi_t$ obtained via the push-forward of~$\widehat\Psi$ through the flow map 
associated with the ODE 
\begin{equation}\label{ODE}
(\dot x,\dot\lambda)=b_{\Psi_t}(x,\lambda)
\end{equation}
is called a Lagrangian solution.
Since Lagrangian solutions are also Eulerian solutions, the equivalence of the two notions follows if one is able to prove that Eulerian solutions are also Lagrangian. For the model studied in~\cite{MorSol19}, and also for other relevant ones~\cite{CCR2011}, these two notions of solution are equivalent. This has been achieved by means of the \emph{superposition principle} (see~\cite{Smirnov93}, and also \cite[Theorem~8.2.1]{AmbGigSav08}, \cite[Theorem~7.1]{AmbTre14}, and \cite[Theorem~5.2]{AmbForMorSav18}), which provides the uniqueness of Eulerian solutions \cite[Theorem~5.3]{AmbForMorSav18}.
Furthermore, the Lagrangian formulation has been used to propose discretization schemes to solve the nonlinear PDE numerically~\cite{CMW,lagr-cm,LT17,MP19,Pla19}. 

Moreover, the Lagrangian point of view has been used in \cite{AmbForMorSav18} to provide a heuristic derivation of the nonlinear continuity equation arising as the mean-field limit of the spatially inhomogeneous replicator dynamics.
Let us briefly discuss this derivation. Denoting by~$h=T/N$ the time step, if an agent at time~$t=ih$, for~$i\in\{0,\ldots,N-1\}$, is in the
position~$x$ with mixed strategy~$\lambda$, first they optimize the strategy distribution following a homogeneous replicator dynamics of the form
\begin{equation}\label{e:lambdaprimo}
\lambda'\coloneqq \lambda+h\T_{\Psi_{t}}(x,\lambda) \,.
\end{equation}
Here,~$\T_{\Psi_t}(x,\lambda)$ is the payoff operator determining the enhancement or suppression of the strategies; it depends on the random state~$(x,\lambda)$  and also on the current distribution~$\Psi_t$. In the setting of \cite{AmbForMorSav18}, the operator~$\T$ is quadratic in~$\lambda$.
After updating the strategy portfolio, the agent updates its position $x$ to 
\begin{equation}\label{e:xprimo}
x'\coloneqq x+h v(x,u) \,,
\end{equation}
choosing~$u$ with probability~$\lambda'$. 
The two equations above completely determine the conditional probability of having an agent in a state~$(x',\lambda')$ at time~$t+h$ given the distribution~$\Psi_t$. 
Equivalently, the new distribution~$\Psi_{t+h}$ can be defined via duality by
\[\begin{split}
\int_{\R^d\times\Pp(U)} & \phi(x',\lambda')\,\di\Psi_{t+h}(x',\lambda')=\! \int_{\R^d\times\Pp(U)} \!\! \bigg(\int_U \phi(x+hv(x,u),\lambda+h\T_{\Psi_t}(x,\lambda))\,\di\lambda'(u) \bigg)\di\Psi_t(x,\lambda) 
\end{split}\]
where~$\phi\colon \R^d\times\Pp(U)\to\R$ is of class~$C^1$.
By a formal first-order Taylor expansion, we have
$$\int_{\R^d\times\Pp(U)}  \phi(x',\lambda')\,\di\Psi_{t+h}(x',\lambda')=  \int_{\R^d\times\Pp(U)} \big[\phi(x,\lambda)+h \nabla\phi(x,\lambda)\cdot b_{\Psi_t}(x,\lambda)\big]\,\di\Psi_t(x,\lambda)+o(h),$$
where 
$$b_{\Psi_t}(x,\lambda)=\left(\begin{array}{cc}
\displaystyle \int_U v(x,u)\,\di\lambda(u) \\ [1mm]
\T_{\Psi} (x,\lambda)
\end{array}\right).$$
In the formal limit for~$h\to0$, we obtain the weak formulation of the nonlinear continuity equation \eqref{NCE}.
A related heuristic derivation has been outlined also in \cite[Remark~4.1]{ABRS}, in the context of a leader-follower dynamics which also fits in the setting of \cite{MorSol19}. In this case, the~$\R^d$-component of~$b_\Psi$ also depends on~$\Psi$, whereas the~$\lambda$-component acts linearly on~$\lambda$, modelling a Markov chain on~$U$.

\smallskip

\paragraph{\textbf{Results of this paper.}}
The main objective of this paper is to present a rigorous proof of the formal derivation described above, by means of an alternate Lagrangian scheme.
The scheme we propose is suitable for approximating all equations in the class considered in \cite{MorSol19} (we refer to Section~\ref{s:hp} for the precise details).
The method is a \emph{Lagrangian} one as it is based on the approximation of the ODE \eqref{ODE}, and it is \emph{alternate} because the updates of~$x$ and~$\lambda$ do not happen simultaneously, but follow the heuristics described above. Indeed, first we make an incremental step in~$\lambda$ and then use the updated~$\lambda'$ to make the incremental step in~$x$.

Since the velocity field~$b$ depends explicitly on~$\Psi$, at each incremental step the updates of~$x$,~$\lambda$, and of the distribution~$\Psi$ involve three substeps, which are the rigorous formalization of the heuristics discussed above. To be precise,
\begin{itemize}\compresslist
\item first we update~$\lambda$ to~$\lambda'$ in the spirit of~\eqref{e:lambdaprimo} (see~\eqref{e:ODE1});
\item then we transport~$\lambda'$ to the state of the system~$\widetilde \Psi$ (see \eqref{e:tilde_Psi}). This amounts to assuming that all the agents know the optimal label distribution~$\lambda'$ of the other agents;
\item then we update the positions~$x$ to~$x'$ in the spirit of~\eqref{e:xprimo} where the velocity field depends on~$\widetilde\Psi$ (see~\eqref{e:ODE2}). Notice that, in our general framework, the velocity field depends on~$\Psi$ and this makes the previous step necessary;
\item finally, we update the global distribution to~$\Psi'$ keeping both~$x'$ and~$\lambda'$ into account (see~\eqref{e:Psi}).
\end{itemize}
Our first main result is Theorem~\ref{t:limit} on the convergence of the scheme presented above.

In Sections~\ref{s:alternative} and~\ref{s:markov}, we turn our attention to the case of the inhomogeneous replicator dynamics considered in \cite{AmbForMorSav18} and to the leader-follower-type dynamics of \cite[Section~5.1]{MorSol19}, respectively. 
More in general, for the second case, we assume that~$\T_\Psi(x,\lambda)$ is a Markov chain on a finite space of an arbitrary number~$n$ of labels.

In the spatially homogeneous case, that is, when there is no~$x$ dependence in the vector field~$b$, in both situations the evolutions of the~$\lambda$-components are gradient flows of suitable energies with respect to certain metric structures, and the solution can be approximated via a minimizing movement scheme~\cite{AmbGigSav08,JKO}. 
The spatially homogeneous replicator equation is a gradient flow with respect to the spherical Hellinger distance~\eqref{e:arccos} of probability measures (this could be obtained, for instance, for a proper choice of~$f$ in \cite[formula~(1.8)]{KonVor19}). 
The spatially homogeneous Markov-type jump processes are the gradient flow of an entropy-like energy penalized by a distance induced by the transition matrix~\cite{Maas,Mielke}.

We investigate the compliance of these structures with our scheme. More precisely, the explicit step~\eqref{e:lambdaprimo} is replaced by an implicit one, which is a minimizing movement step suggested by the aforementioned gradient flow structure (see~\eqref{e:HS_step} and~\eqref{e:Markov_step}, respectively).
A relevant difficulty in the spatially inhomogeneous setting is that the energy and the dissipation distances that we consider may as well depend on the state~$\Psi$, which changes from step to step.
This extension is far from trivial and requires a careful analysis of the related Euler conditions, which is partially inspired by \cite[Section~4.2]{GalMon17} for the case of the replicator dynamics.
This is done is Propositions~\ref{p:HS_euler} and~\ref{p:Markov_euler}, respectively, where we show that the deviation from the explicit scheme is uniformly controlled by the vanishing time step.

The two main results of Sections~\ref{s:alternative} and~\ref{s:markov} are given by Theorems~\ref{t:limit2} and~\ref{t:Markov}, proving the convergence of our alternate Lagrangian scheme to the unique solution to~\eqref{NCE}. In particular, Theorem~\ref{t:limit2} is a global-in-time convergence result for the spatially inhomogeneous replicator dynamics, whereas Theorem~\ref{t:Markov} provides a short-time existence result for a well-prepared initial datum for spatially inhomogeneous Markov-type jump processes.

\smallskip

The paper is structured as follows: in Section~\ref{s:hp} we introduce the structural assumptions on the systems that we consider. In Section~\ref{s:scheme} we describe the alternate Lagrangian scheme, which we apply to the inhomogeneous replicator dynamics in Section~\ref{s:alternative} and to the inhomogeneous Markov-type jump processes in Section~\ref{s:markov}.

\section{The mathematical setting}\label{s:hp}

\paragraph{\textbf{Basic notation.}} 
Given a metric space~$(X,\sfd_X)$, we denote by~$\mathcal{M}(X)$ the space of signed Borel measures~$\mu$ in~$X$ with finite total variation~$\| \mu \|_{\mathrm{TV}}$,
by~$\mathcal{M}_+(X)$ and~$\cP(X)$ the convex subsets of nonnegative measures and probability measures, respectively. 
We say that~$\mu \in \cP_c(X)$ if~$\mu \in \Pp(X)$ and the support~$\spt \, \mu$ of~$\mu$ is a compact subset of~$X$. Moreover, for~$K \subseteq X$ we will use the notation~$\Pp(K)$ to indicate the set of measures~$\mu \in \Pp(X)$ such that~$\spt\, \mu \subseteq K$.

As usual, if~$(Z, \sfd_{Z})$ is another metric space, for every~$\mu\in\mathcal{M}_+(X)$ and every~$\mu$-measurable function~$f\colon X\to Z$, we define the push-forward measure~$f_\#\mu\in\mathcal{M}_+(Z)$ by~$(f_\#\mu)(B)\coloneqq \mu(f^{-1}(B))$ for any Borel set~$B\subset Z$.
The push-forward measures has the same total mass as~$\mu$, namely~$\mu(X)=(f_\#)\mu(Z)$.



For a Lipschitz function~$f\colon X\to\mathbb{R}$ we set
\begin{equation*}
\mathrm{Lip}(f)\coloneqq\sup_{x,y\in X \atop x\neq y}\frac{|f(x)-f(y)|}{\sfd_X(x,y)}
\end{equation*}
its Lipschitz constant. We denote by~$\mathrm{Lip}(X)$ and~$\mathrm{Lip}_b(X)$  the spaces of Lipschitz and bounded Lipschitz functions on~$X$, respectively. Both are normed spaces with the norm~$\lVert f\rVert_{\mathrm{Lip}} \coloneqq \lVert f\rVert_\infty+ \mathrm{Lip}(f)$, where~$\lVert\cdot\rVert_\infty$ is the supremum norm. Furthermore, we use the notation~$\mathrm{Lip}_{1}(X)$ for the set of functions~$f \in \mathrm{Lip}_{b} (X)$ such that~$\mathrm{Lip}(f) \leq 1$.

In a complete and separable metric space~$(X,\sfd_X)$,  we shall use the Kantorovich-Rubinstein distance~$W_1$ in the class~$\cP(X)$, defined as
\begin{equation*}
W_1(\mu,\nu)\coloneqq\sup\bigg\{\int_X\varphi\,\mathrm{d}\mu-\int_X\varphi\,\mathrm{d}\nu: \varphi\in\mathrm{Lip}_1(X) \bigg\} \,.
\end{equation*}
Notice that~$W_1(\mu,\nu)$ is finite if~$\mu$ and~$\nu$ belong to the space
\begin{equation*}
\cP_1(X)\coloneqq \bigg\{\mu\in\cP(X):  \text{$\int_X \sfd_X(x,\bar x)\,\mathrm{d}\mu(x)<+\infty$ for some $\bar x\in X$}\bigg\}
\end{equation*}
and that~$(\cP_1(X),W_1)$ is complete if~$(X,\sfd_X)$ is complete.

If~$(E, \| \cdot\|_{E})$ is a Banach space and~$\mu \in\mathcal{M}_+(E)$, we define the first moment~$m_1(\mu)$ as
\begin{equation*}
m_1(\mu)\coloneqq\int_{E} \lVert x \rVert_E\,\mathrm{d}\mu\,.
\end{equation*}
Notice that, for a probability measure~$\mu$, finiteness of the integral above is equivalent to~$\mu \in \cP_1(E)$, whenever~$E$ is endowed with the distance induced by the norm~$\lVert\cdot  \rVert_{E}$.

For a Banach space~$E$, the notation~$C^1_b(E)$ will be used to denote the subspace of~$C_b(E)$ of functions having bounded continuous Fr\'echet differential at each point. The notation~$\nabla\phi(\cdot)$ will be used to denote the Fr\'echet differential. In the case of a function~$\phi \colon [0,T]\times E \to \mathbb{R}$, the symbol~$\partial_t$ will be used to denote partial differentiation with respect to~$t$. The symbol $\langle \cdot, \cdot \rangle$ will be used to denote duality products, with no  further specification if the meaning is clear from the context.
\smallskip

\paragraph{\textbf{Functional setting.}} We consider a set of pure strategies~$U$, where~$U$ is a compact metric space, and we denote by~$Y\coloneqq \R^{d} \times \Pp(U)$ the state-space of the system. Precisely, for every $y= (x, \lambda) \in Y$, the component $x\in\R^{d}$ describes the location of an agent in space, whereas the component $\lambda\in\cP(U)$ describes the distribution of labels of the agent.

The correct functional space for the dynamics (see also~\cite{AmbForMorSav18, MorSol19}) is the space $\overline{Y}\coloneqq \R^{d} \times \F(U)$, where we have set (see, e.g.,~\cite{AP,MR81458} and \cite[Chapter~3]{MR3792558})
\begin{equation}\label{P1}
\F(U) \coloneqq \overline{ \spann ( \Pp (U) ) }^{\|\cdot\|_{\mathrm{BL}}} \subseteq (\mathrm{Lip}(U))'.
\end{equation}
The closure in~\eqref{P1} is taken with respect to the \emph{bounded Lipschitz} norm $\lVert\cdot\rVert_{\BL}$, defined as
\begin{equation*}
\lVert \mu \rVert_{\BL}\coloneqq\sup \big\{\langle \mu,\varphi\rangle:  \varphi\in \Lip(U),  \|\varphi\|_{\Lip}\leq 1\big\} \qquad \text{for every $\mu\in(\Lip(U))'$}\,.
\end{equation*}
We notice that, by definition of~$\| \cdot\|_{\mathrm{BL}}$, we always have
\begin{displaymath}
\| \mu \|_{\mathrm{BL}} \leq \| \mu \|_{\mathrm{TV}} \qquad \text{for every $\mu \in \mathcal{M}(U)$}\,.
\end{displaymath}
In particular, $\| \lambda \|_{\mathrm{BL}} \leq 1$ for every $\lambda \in \Pp(U)$.

We endow $\overline{Y}$ with the norm
\begin{displaymath}
\lVert y\rVert_{\overline Y}=\lVert(x,\lambda)\rVert_{\overline Y}\coloneqq \lvert x \rvert+\lVert \lambda \rVert_{\BL} \,.
\end{displaymath}
For every $R>0$, we denote by~$\B_R$ the closed ball of radius~$R$ in $\R^{d}$ and by $\B_R^Y$ the ball of radius~$R$ in~$Y$, namely $\B_R^Y=\{y\in Y:\lVert y\rVert_{\overline Y}\leq R\}$. We notice that~$\B^{Y}_{R}$ is a compact set, as~$Y$ is locally compact by our assumptions on~$U$.

As in~\cite{MorSol19}, we consider, for every~$\Psi \in \Pp_{1} (Y)$, the velocity field $v_{\Psi} \colon Y \to \R^{d}$ such that
\begin{itemize}
\item[($v_1$)] for every $R>0$, $v_{\Psi} \in \mathrm{Lip} (\B^{Y}_{R}; \R^{d})$ uniformly with respect to~$\Psi \in \Pp( \B^{Y}_{R})$, i.e., there exists~$L_{v, R}>0$ such that
\begin{displaymath}
| v_{\Psi} (y_{1}) - v_{\Psi}(y_{2}) | \leq L_{v, R} \| y_{1} - y_{2} \|_{\overline{Y}} \qquad \text{for every $y_{1}, y_{2} \in Y$}\,;
\end{displaymath}

\item[($v_2$)] for every $R>0$ there exists $L_{v, R}>0$ such that for every $\Psi_{1}, \Psi_{2} \in \Pp(\B^{Y}_{R})$ and every $y \in \B^{Y}_{R}$
\begin{displaymath}
| v_{\Psi_{1}} (y) - v_{\Psi_{2}} (y) | \leq L_{v, R} W_{1} (\Psi_{1}, \Psi_{2})\,;
\end{displaymath}

\item[($v_3$)] there exists $M_{v}>0$ such that for every $y \in Y$ and every $\Psi \in \Pp_{1}(Y)$
\begin{displaymath}
| v_{\Psi} (y) | \leq M_{v} \big( 1 + \| y \|_{\overline{Y}} + m_{1} ( \Psi) \big) \,.
\end{displaymath}
\end{itemize}

As for~$\T$, for every $\Psi \in  \Pp_{1} (Y)$ we assume that the operator $\T_{\Psi} \colon Y \to \F(U)$ is such that
\begin{itemize}

\item[($\T_0$)] for every $(y , \Psi) \in Y \times \Pp_{1}( Y )$, the constants belong to the kernel of~$\T_{\Psi}(y)$, i.e., 
\begin{displaymath}
\left\langle \T_{\Psi}(y), 1 \right\rangle_{\F(U), \mathrm{Lip}(U)} = 0 \,;
\end{displaymath}

\item[($\T_1$)] there exists $M_{\T}>0$ such that for every $y \in Y$ and every $\Psi \in \Pp_{1}(Y)$
\begin{displaymath}
\| \T_{\Psi}(y) \|_{\mathrm{BL}} \leq M_{\T} \big( 1 + \| y \|_{\overline{Y}} + m_{1} ( \Psi) \big)\,;
\end{displaymath}

\item[($\T_2$)] for every $R>0$, there exists $L_{\T, R}>0$ such that for every $(y_{1}, \Psi_{1}), (y_{2}, \Psi_{2}) \in \B^{Y}_{R} \times \Pp (\B^{Y}_{R})$
\begin{displaymath}
\| \T_{\Psi_{1}} ( y_{1} ) - \T_{\Psi_{2}} (y_{2} ) \|_{\overline{Y}} \leq L_{\T, R} \big( \| y_{1} - y_{2} \|_{\mathrm{BL}} + W_{1} (\Psi_{1}, \Psi_{2}) \big)\,;
\end{displaymath}

\item[($\T_3$)] for every $R>0$ there exists $\delta_{R}>0$ such that for every $(y, \Psi) \in \B^{Y}_{R} \times \Pp_{1}(Y)$ we have
\begin{displaymath}
\T_{\Psi} (y) + \delta_{R} \lambda \geq 0\,.
\end{displaymath}
\end{itemize}

Finally, for every $y  \in Y$ and every $\Psi \in \Pp_{1}(Y)$ we set
\begin{equation}\label{e:b_Psi}
b_{\Psi}(y)  \coloneqq 
\left(\begin{array}{cc}
\displaystyle v_{\Psi}(y) \\ [1mm]
\displaystyle \T_{\Psi} (y)
\end{array}\right),
\end{equation}
which is the velocity field driving the evolution (see \eqref{e:cont_eq} below).


\section{The alternate Lagrangian scheme}\label{s:scheme}

Let $\widehat{\Psi} \in \Pp_{c}(Y)$ be a probability measure on~$Y$ with compact support in~$Y$. Given $T>0$, for every $k\in \mathbb{N}\setminus\{0\}$ we set $\tau_k \coloneq T/k$ and, for $i \in \{0, \ldots, k\}$, $t^{k}_{i} \coloneqq i\tau_{k}$. 

We now show how to construct a curve $\Psi^{k} \colon [0,T] \to \Pp_{1}(Y)$, defined piecewise on each time interval $[t_{i}^{k},t_{i+1}^{}k]$, which approximates a solution~$\Psi \in C([0,1]; \Pp_{1}(Y))$ of the initial-value problem for the nonlinear continuity equation
\begin{equation}\label{e:cont_eq}
\partial_{t} \Psi_{t} + \dive ( b_{\Psi_{t}} \Psi_{t}) = 0 \,, \qquad \Psi_{0} = \widehat{\Psi}\,.
\end{equation}
Let $\Psi^{k}_{0} \coloneqq \widehat{\Psi}$. In each interval $[t^{k}_{i}, t^{k}_{i+1})$, assume the measure $\Psi^{k}_{i} \in \Pp_{1}(Y)$ to be known. With this knowledge, we update the state of the system with the following procedure, consisting of two steps.

{\bf Step 1.} We update the label $\lambda_{(\hat x,\hat\lambda)}(t) \in \Pp(U)$ of a player that at time~$t^{k}_{i}$ sits in~$\hat x \in \R^{d}$ with label~$\hat \lambda \in \Pp(U)$ by setting
\begin{equation}\label{e:ODE1}
\lambda_{(\hat x, \hat\lambda)}(t^k_{i+1})\coloneqq \lambda_{(\hat x,\hat\lambda)}+\tau_k \T_{\Psi_i^k}\big(\hat x,\lambda_{(\hat x,\hat \lambda)}(t^k_i)\big)\,.
\end{equation}
At this stage, we assume that~$\lambda_{(\hat{x}, \hat{\lambda})} (t^{k}_{i+1}) \in \Pp(U)$ and we continue with the construction of the piecewise affine interpolant between~$\lambda_{(\hat x,\hat\lambda)}(t^{k}_{i})$ and~$\lambda_{(\hat x,\hat\lambda)}(t^{k}_{i+1})$, defined as the function~$\lambda^{k}_{(\hat x,\hat\lambda),i+1} \colon [t^{k}_{i}, t^{k}_{i+1}] \to \Pp(U)$ such that
\begin{equation}\label{100}
\lambda^{k}_{(\hat x,\hat\lambda),i+1}(t)\coloneqq \frac{t-t_{i}^{k}}{\tau_k} \lambda_{(\hat x,\hat\lambda)}(t^{k}_{i+1})+ \bigg(1-\frac{t-t_{i}^{k}}{\tau_k}\bigg) \lambda_{(\hat x,\hat\lambda)}(t^{k}_{i})\,.
\end{equation}
In Lemma~\ref{l:Gronwall_bound} below, we show that the assumption~$\lambda_{(\hat{x}, \hat{\lambda})} (t^{k}_{i+1}) \in \Pp(U)$ is actually satisfied for~$k$ large enough (and therefore $\tau_{k}$ small enough), independently of~$i=0, \ldots, k-1$.
Giving Lemma~\ref{l:Gronwall_bound} for granted for the time being, we define the map~$\Lambda^{k}_{i+1} \colon  [t^{k}_{i}, t^{k}_{i+1}] \times \R^{d} \times \Pp(U) \to \Pp(U)$ as
\begin{equation}\label{e:lambda}
\Lambda^{k}_{i+1} (t, \hat x, \hat\lambda) \coloneqq \lambda^{k}_{(\hat x,\hat\lambda),i+1} (t) 
\qquad \text{for every $(t,\hat x, \hat\lambda) \in[t^{k}_{i}, t^{k}_{i+1} ] \times \R^{d} \times \Pp(U)$}\,,
\end{equation}
and transport it to the state of the system by defining
\begin{equation}\label{e:tilde_Psi}
\widetilde{\Psi}^{k}_{i+1} \coloneqq (\id; \Lambda^{k}_{i+1} (t^{k}_{i+1}, \cdot, \cdot))_{\#} \Psi^{k}_{i} \in \Pp_{1}(Y)\,.
\end{equation}

{\bf Step 2.} In the second step we update the positions of the players. Precisely, a player that at time~$t^{k}_{i}$ sits in the position~$\hat x$ with label~$\hat\lambda$ will now move following the velocity field given by~$v_{\widetilde{\Psi}^{k}_{i+1}}\big( x_{(\hat x,\hat\lambda)}(t_{i}^{k}), \lambda^{k}_{(\hat x,\hat\lambda),i+1}(t^{k}_{i+1}) \big)$, which is determined by the updated label~$\lambda^{k}_{(\hat x,\hat\lambda),i+1}(t^{k}_{i+1})$  just obtained in~\eqref{e:ODE1}. Hence, we set
\begin{equation}\label{e:ODE2}
x_{(\hat x,\hat\lambda)}(t_{i+1}^{k}) \coloneqq x_{(\hat x,\hat\lambda)}(t_{i}^{k})+\tau_k v_{\widetilde{\Psi}^{k}_{i+1}} \big( x_{(\hat x,\hat\lambda)}(t_{i}^{k}), \lambda^{k}_{(\hat x,\hat\lambda),i+1}(t^{k}_{i+1}) \big) \,.
\end{equation}
Also in this case, 
we can define the affine interpolant between~$x_{(\hat x,\hat\lambda)}(t^{k}_{i})$ and~$x_{(\hat x,\hat\lambda)}(t^{k}_{i+1})$, as a function~$x^{k}_{(\hat x,\hat\lambda),i+1} \colon [t^{k}_{i}, t^{k}_{i+1}] \to \R^{d}$, by
\begin{equation}\label{101}
x^{k}_{(\hat x,\hat\lambda),i+1}(t)\coloneqq \frac{t-t_{i}^{k}}{\tau_k} x_{(\hat x,\hat\lambda)}(t^{k}_{i+1})+ \bigg(1-\frac{t-t_{i}^{k}}{\tau_k}\bigg) x_{(\hat x,\hat\lambda)}(t^{k}_{i})
\end{equation}
We notice that~\eqref{101}, in contrast with~\eqref{100}, is always well defined, since~$\R^{d}$ is a convex space and the velocity field is an element of~$\R^{d}$.

Eventually, we define the map $X^{k}_{i+1} \colon [t^{k}_{i}, t^{k}_{i+1}] \times \R^{d} \times \Pp(U) \to \R^{d}$ as
\begin{equation}\label{e:X}
X^{k}_{i+1} (t,\hat x, \hat \lambda) \coloneqq x^{k}_{(\hat x,\hat\lambda),i+1} (t) \qquad \text{for every $(t, \hat x, \hat\lambda) \in [t^{k}_{i}, t^{k}_{i+1}] \times \R^{d} \times \Pp(U)$}
\end{equation}
and we set
\begin{equation}\label{e:Psi}
\Psi^{k} (t) \coloneqq \Big( X^{k}_{i+1} (t, \cdot, \cdot) ; \Lambda^{k}_{i+1} (t, \cdot, \cdot ) \Big)_{\#} \Psi^{k}_{i}\,, \qquad \Psi^{k}_{i+1} \coloneqq \Psi^{k} (t^{k}_{i+1})\,.
\end{equation}
For later use, we also define
\begin{eqnarray}
&& \label{e:Psi_tilde}
\widetilde{\Psi}^{k}(t) \coloneqq \widetilde{\Psi}^{k}_{i+1} \qquad \text{for every $t\in (t^{k}_{i}, t^{k}_{i+1}]$}\,,\\[1mm]
&& \label{e:underline_Psi}
 \underline{\Psi}^{k}(t) \coloneqq \Psi^{k}_{i} \qquad \text{for every $t\in [t^{k}_{i}, t^{k}_{i+1})$}\,.
\end{eqnarray}

By an application of Gronwall inequality, in the following lemma we give an estimate of $\Big| x^{k}_{(\hat x,\hat\lambda),i+1} (t)\Big |$ and $\Big\| \lambda^{k}_{(\hat x,\hat\lambda),i+1}( t)\Big\|_{\mathrm{BL}}$ in terms of $|\hat x |$ and~$\|\hat \lambda \|_{\mathrm{BL}}$. As a consequence, we deduce that the construction above is well defined for $\tau_k$ sufficiently small and can be iterated over~$i= 0, \ldots, k-1$, since the initial condition~$\widehat{\Psi}$ has a compact support in~$Y$. This indeed implies that each~$\Psi^{k}_{i}$ belongs to~$\Pp_{c}(Y) \subseteq \Pp_{1}(Y)$. 


\begin{lemma}\label{l:Gronwall_bound}
Let $\widehat{\Psi} \in \Pp_{c}(Y)$. Then, for~$k$ large enough the curves~$\Psi^{k}(\cdot)$,~$\underline{\Psi}^{k}(\cdot)$, and~$\widetilde{\Psi}^{k}(\cdot)$ are well defined from~$[0,T]$ with values in~$\Pp_{1}(Y)$. Furthermore, there exists~$R>0$ independent of~$k$ and~$t$ such that $\Psi^{k}(t), \underline{\Psi}^{k}(t), \widetilde{\Psi}^{k}(t)\in \Pp( \B^{Y}_{R})$.
\end{lemma}

\begin{proof}
Along the proof of the lemma we denote with~$\lambda^{k}(t, x_{0}, \lambda_{0})$ and~$x^{k}(t, x_{0}, \lambda_{0})$, for $(x_{0}, \lambda_{0}) \in \spt \widehat{\Psi} = : \mathcal{S}$,  the curves obtained by iteratively solving the difference equations~\eqref{e:ODE1} and~\eqref{e:ODE2} in each interval~$[t^{k}_{i}, t^{k}_{i+1}]$ starting from~$(x_{0}, \lambda_{0})$ at time $t_0=0$ and using, at each node~$t^{k}_{i}$, $i= 1, \ldots, k-1$, $\hat \lambda = \lambda^{k}(t^{k}_{i}, x_{0}, \lambda_{0})$ and~$\hat x = x^{k}(t^{k}_{i}, x_{0}, \lambda_{0})$ as new initial conditions. 

As we have already noticed above, the curve~$x^{k}(t, x_{0}, \lambda_{0})$ is well-defined as long as~$\lambda^{k}(t, x_{0}, \lambda_{0})$ and the measures~$\widetilde{\Psi}^{k}_{i}$ are. Therefore, in order to prove the lemma it is enough to show that, for~$\tau_{k}$ small enough, for every~$(x_0, \lambda_0) \in \spt\, \widehat\Psi$ the piecewise linear interpolant~$\lambda^{k} (t, x_{0}, \lambda_{0})$ always belongs to~$\Pp(U)$. This can be done recursively by arguing on each interval~$[t^{k}_{i}, t^{k}_{i+1}]$,~$i=0, \ldots, k-1$. 

To simplify our estimates, we define the piecewise constant interpolation functions
\begin{equation}\label{e:110}
\begin{aligned}
& \underline{x}^{k} (t, x_{0}, \lambda_{0}) \coloneqq x^{k} (t^{k}_{j}, x_{0}, \lambda_{0})\,,  \quad \underline{\lambda}^{k} (t, x_{0}, \lambda_{0}) \coloneqq \lambda^{k}(t^{k}_{j}, x_{0}, \lambda_{0}) \quad & \text{for $t \in [t^{k}_{j}, t^{k}_{j+1})$}\,,\\
& \overline{\lambda}^{k}(t, x_{0}, \lambda_{0}) \coloneqq \lambda^{k}(t^{k}_{j+1}, x_{0}, \lambda_{0}) & \text{for $t \in (t^{k}_{j}, t^{k}_{j+1}]$}\,.
\end{aligned}
\end{equation}
For $i=0$ we have that the initial condition~$\lambda_{0} \in \Pp(U)$, hence there is nothing to show. Assuming that $\lambda^{k} (t^{k}_{j}, x_{0}, \lambda_{0}) \in \Pp(U)$ for every $j =0, \ldots, i$ and every $(x_0, \lambda_0) \in \spt\, \widehat{\Psi}$, we show that $\lambda^{k} (t^{k}_{i+1}, x_{0}, \lambda_{0}) \in \Pp(U)$ for~$k$ large enough, independently of~$i$ and of the initial condition~$(x_0, \lambda_0)$. Since, recalling~\eqref{e:ODE1} and \eqref{100}, we define
\begin{displaymath}
\lambda^{k}( t , x_0 , \lambda_{0}) \coloneqq \underline{\lambda}^{k} (t, x_{0}, \lambda_{0}) + (t - t^{k}_{i}) \T_{\underline \Psi^{k}( t ) }(\underline{x}^{k}(t, x_{0}, \lambda_{0}), \underline{\lambda}^{k}(t , x_{0}, \lambda_{0})) \qquad \text{for $t \in [t^{k}_{i}, t^{k}_{i+1}]$} \,;
\end{displaymath}
 by assumptions~$(\T_{0})$ and~$(\T_{3})$ we are led to showing that the piecewise constant interpolation functions~$\underline{x}^{k}( t, x_{0}, \lambda_{0})$ and~$\underline{\lambda}^{k}( t, x_{0},\lambda_{0})$ are bounded in~$\R^{d}$ and~$\F(U)$, respectively, uniformly with respect to~$(x_{0}, \lambda_{0}) \in \mathcal{S}$ and $t \in [0,t^{k}_{i+1}]$, and that the bound does not depend on~$i$. 
Indeed, if this is the case, let~$R' > 0$ be such that $(\underline{x}^{k}(t, x_{0}, \lambda_{0}) , \underline{\lambda}^{k} ( t, x_{0}, \lambda_{0})) \in \B^{Y}_{R'}$ for every $t \in [0, t^{k}_{i+1}]$ and every $(x_{0}, \lambda_{0}) \in \mathcal{S}$. In particular, by construction~\eqref{e:underline_Psi} of~$\underline\Psi^{k}(t)$ it holds~$\underline\Psi^{k}(t) \in \Pp(\B^{Y}_{R'})$. By~$(\T_3)$ there exists~$\delta_{R'}>0$, independent of~$k$,~$i$, and~$(x_0, \lambda_0) \in \mathcal{S}$, such that for $t \in [t^{k}_{i}, t^{k}_{i+1}]$
\begin{displaymath}
\lambda_{R' } \coloneqq \frac{1}{\delta_{R'}} \T_{\underline{\Psi}^{k}(t)} \big( \underline{x}^{k}(t, x_{0}, \lambda_{0}), \underline{\lambda}^{k}(t, x_{0}, \lambda_{0}) \big) +  \underline{\lambda}^{k}(t, x_{0}, \lambda_{0}) \geq 0\,.
\end{displaymath}
In particular, assumption~$(\T_1)$ implies that~$\lambda_{R'} \in \F(U)$ and satisfies
\begin{displaymath}
\big | \left \langle \lambda_{R'} , \eta \right\rangle_{\F(U), \mathrm{Lip}(U)} \big | \leq \| \eta\|_{\infty} \| \lambda_{R'} \|_{\mathrm{BL}}\,,
\end{displaymath}
so that~$\lambda_{R'}$ can be extended in a unique way to a linear and continuous operator on~$C(U)$. 
The Riesz representation theorem yields that $\lambda_{R'} \in \mathcal{M}_{+}(U)$. Moreover, by~$(\T_0)$ we get
\begin{displaymath}
\left\langle \lambda_{R'}, 1 \right\rangle_{\F(U), \mathrm{Lip}(U)} = \big\langle \underline{\lambda}^{k}(t, x_{0}, \lambda_{0}) , 1 \big\rangle_{\F(U), \mathrm{Lip}(U)} = 1\,,
\end{displaymath}
which implies $\lambda_{R'} \in \Pp(U)$. By the convexity of~$\Pp(U)$ we deduce that whenever $\tau_{k} \leq 1/ \delta_{R'}$
\begin{displaymath}
\lambda^{k}(t, x_{0}, \lambda_{0}) = \underline\lambda^{k}(t, x_{0}, \lambda_{0}) + ( t - t^{k}_{i})  \T_{\underline{\Psi}(t^{k}_{i})} ( \underline{x}^{k}( t, x_{0}, \lambda_{0}) , \underline{\lambda}^{k}( t, x_{0}, \lambda_{0})) \in \Pp(U)
\end{displaymath}
for every $t \in [t^{k}_{i}, t^{k}_{i+1}]$. Being the upper bound~$R'$ independent of~$i$ and of~$(x_{0}, \lambda_{0}) \in \mathcal{S}$, also~$\delta_{R'}$ is. 
Hence, the trajectories $x^{k}(\cdot, x_0, \lambda_0)$ and $\lambda^{k}(\cdot, x_0, \lambda_0)$ are well defined from~$[0,T]$ with values in~$\R^{d}$ and~$\Pp(U)$, respectively.


In order to conclude that the interpolation curves~$x^{k} (t, x_{0}, \lambda_{0})$ and~$\lambda^{k}(t, x_{0}, \lambda_{0})$ are well-defined, we have to estimate $|\underline{x}^{k}(t, x_{0}, \lambda_{0})|$ and~$\lVert\underline{\lambda}^{k}(t, x_{0}, \lambda_{0})\rVert_{\BL}$ for $(x_0, \lambda_{0}) \in \mathcal{S}$. Since we are assuming that~$\lambda^{k}(t^{k}_{j}, x_{0}, \lambda_{0}) \in \Pp(U)$ for $j \in 0, \ldots, i$, we have that $\| \lambda^{k}(t^{k}_{j} , x_{0}, \lambda_{0})\|_{\mathrm{BL}} \leq 1$, and the same holds for~$\lVert \underline{\lambda}^{k}(t, x_{0}, \lambda_{0}) \rVert_{\BL}$. As for~$\underline{x}^{k}(t, x_{0}, \lambda_{0})$, using~\eqref{e:ODE2} and~$(v_3)$ we get
\begin{align}\label{e:104}
| \underline{x}^{k}(t, x_{0}, \lambda_{0})| & \leq | x_{0} | + \int_{0}^{t^{k}_{i}} \big | v_{\widetilde{\Psi}^{k}(\tau)} ( \underline{x}^{k}(\tau, x_{0}, \lambda_{0}), \overline{\lambda}^{k}( \tau, x_{0}, \lambda_{0}) \big | \, \di \tau
\\
&
\leq |x_{0}| + \int_{0}^{t} M_{v} \Big( 3 + 2 \sup_{( \hat{x}, \hat{\lambda}) \in \mathcal{S}} \, | \underline{x}_{k}( \tau, \hat{x}, \hat{\lambda})| \Big) \, \di \tau\,. \nonumber
\end{align}
Let us now fix~$r>0$ such that $\mathcal{S} \subseteq \B^{Y}_{r}$ and let
\begin{displaymath}
f_{k}( t ) \coloneqq  \sup_{(\hat{x}, \hat{\lambda} ) \in \mathcal{S}} \, | \underline{x}^{k}( t, \hat{x}, \hat \lambda ) | \,.
\end{displaymath} 
By taking the supremum over~$\mathcal{S}$ in~\eqref{e:104} we deduce that
\begin{equation}\label{e:105}
f_{k}(t) \leq r + \int_{0}^{t} 3 M_{v} (1 +  f_{k}(\tau))\, \di \tau\,.
\end{equation}
Applying the Gronwall inequality to~\eqref{e:105} we infer that
\begin{equation}\label{e:106}
f_{k}(t) \leq (r + 3 M_{v}  T) e^{3 M_{v}  T} \,.
\end{equation} 
Setting~$R' \coloneqq 1 + (r + 3 M_{v}  T) e^{3 M_{v}  T}$ we have proved that the piecewise constant interpolation function~$t \mapsto (\underline{x}^{k}(t, x_{0}, \lambda_{0}) , \underline{\lambda}^{k}(t, x_{0}, \lambda_{0}))$ belongs to~$\B^{Y}_{R'}$ for every $t \in [t^{k}_{i}, t^{k}_{i+1})$ and every~$(x_{0}, \lambda_{0}) \in \mathcal{S}$. In particular, we notice that the computations above are  independent of the choice of~$i$, as long as we know that~$\lambda^{k}(t^{k}_{j}, x_{0}, \lambda_{0}) \in \Pp(U)$ for every $j=0, \ldots, i$ and every~$(x_{0}, \lambda_{0}) \in \mathcal{S}$. With this control at hand, we conclude, as explained above, that~\eqref{e:ODE1} and~\eqref{e:ODE2} are well-posed.

Finally, we estimate~$x^{k}(t, x_{0}, \lambda_{0})$. For $(x_{0}, \lambda_{0}) \in \spt \widehat{\Psi}$ and $t \in [0,T]$, by~$(v_{3})$ we have
\begin{equation}\label{e:107}
\begin{split}
| x^{k}(t, x_{0}, \lambda_{0}) & | \leq | x_{0} | + \int_{0}^{t} \big| v_{\widetilde{\Psi}^{k}(\tau)} \big( \underline{x}^{k}(\tau, x_{0}, \lambda_{0} ) , \overline{\lambda}^{k}(\tau, x_{0}, \lambda_{0}) \big) \big| \, \di \tau 
\\
&
\vphantom{\int} \leq r + 2M_{v}( 1 + R' ) T\,.
\end{split}
\end{equation}
Setting $R \coloneqq \max \{R', r + 2M_{v} (1 + R')T + 1 \}$, we obtain that $\Psi^{k}(t), \widetilde{\Psi}^{k}(t), \underline{\Psi}^{k} (t) \in \Pp(\B^{Y}_{R})$ for every $t \in [0,T]$ and every $k \in \mathbb{N}$ large enough.
\end{proof}

In the next proposition we show that the curve~$\Psi^{k}(\cdot)$ solves the continuity equation \eqref{e:cont_eq} up to an error of order~$\tau_{k}$. 
\begin{proposition}\label{p:discrete_equation}
Let $\widehat{\Psi} \in \Pp_{c}(Y)$, let~$\Psi^{k} \colon [0,T] \to \Pp_{1}(Y)$ be the curve defined in~\eqref{e:Psi} starting from~$\widehat{\Psi}$, and let~$\widetilde{\Psi}^{k}$ be as in~\eqref{e:Psi_tilde}. Then, the following holds: there exists a positive constant~$C$ such that for every $\varphi \in C_{b}^{1}( \R^{d} \times \F(U))$, every~$k \in \mathbb{N}$, every $i \in \{0, \ldots, k-1\}$, and every $t \in (t^{k}_{i}, t^{k}_{i+1})$,
\begin{equation}\label{e:approx_equation}
\frac{\di}{\di t} \int_{Y} \varphi (x, \lambda) \, \di \Psi^{k}(t)(x, \lambda) = \int_{Y} \nabla \varphi (x, \lambda) \cdot b_{\Psi^{k}(t)} (x, \lambda) \, \di \Psi^{k}(t) (x, \lambda) + \vartheta_{k}(\varphi) \,,
\end{equation}
where $|\vartheta_{k}(\varphi)|  \leq C \| \varphi \|_{C^{1}_{b}} \tau_{k}$.
\end{proposition}

\begin{proof}
Let us fix~$\varphi \in C_{b}^{1}(\R^{d}\times \F(U))$ and~$t \in (t^{k}_{i}, t^{k}_{i+1})$. By definition of~$\Psi^{k}(t)$ we have that
\begin{equation}\label{e:1}
\begin{split}
\frac{\di}{\di t} & \int_{Y} \varphi (x, \lambda) \, \di \Psi^{k}(t)(x, \lambda) = \frac{\di}{\di t} \int_{Y} \varphi \big ( X^{k}_{i+1}( t, x, \lambda ) , \Lambda^{k}_{i+1} ( t, x, \lambda) \big) \, \di \Psi^{k}_{i} (x, \lambda) 
\\
&
= \int_{Y} \nabla_{x} \varphi \big ( X^{k}_{i+1}( t, x, \lambda ) , \Lambda^{k}_{i+1} ( t, x, \lambda) \big) \cdot v_{\widetilde{\Psi}^{k} (t)} \big(  x , \Lambda^{k}_{i+1} (t^{k}_{i+1}, x, \lambda) \big) \, \di \Psi^{k}_{i} (x, \lambda)
\\
&
\quad + \int_{Y} \nabla_{\lambda} \varphi  \big ( X^{k}_{i+1}( t, x, \lambda ) , \Lambda^{k}_{i+1} ( t, x, \lambda) \big) \cdot \T_{\underline{\Psi}^{k}(t)}  ( x, \lambda ) \, \di \Psi^{k}_{i} (x, \lambda) \,,
\end{split}
\end{equation}
where $\widetilde{\Psi}^{k}(t)$ and~$\underline{\Psi}^{k}(t)$ are defined in~\eqref{e:Psi_tilde} and~\eqref{e:underline_Psi}, respectively.
 In order to obtain~\eqref{e:approx_equation} from~\eqref{e:1} we have to estimate the following quantities:
\begin{eqnarray*}
&& \displaystyle I_{1} (x, \lambda) \coloneqq \Big| v_{\widetilde{\Psi}^{k}(t)} \big(  x , \Lambda^{k}_{i+1} (t^{k}_{i+1}, x, \lambda) \big) - v_{\Psi^{k}(t)} \big(  X^{k}_{i+1} (t, x, \lambda), \Lambda^{k}_{i+1} (t , x, \lambda) \big) \Big|\,, \\[2mm]
&& \displaystyle I_{2} (x, \lambda) \coloneqq \Big \| \T_{\underline{\Psi}^{k}(t)} ( x, \lambda ) - \T_{\Psi^{k}(t)} (X^{k}_{i+1} (t, x, \lambda) , \Lambda^{k}_{i+1}(t, x, \lambda) )  \Big \| _{\mathrm{BL}} 
\end{eqnarray*}
for $(x, \lambda) \in \spt  \Psi^{k}_{i} \subseteq \B^{Y}_{R}$, where~$R$ has been determined in Lemma~\ref{l:Gronwall_bound}.

Let us start with~$I_{1}$. By triangle inequality we have
\begin{equation}\label{e:2}
\begin{split}
I_{1} (x, \lambda) & \leq   \Big| v_{\widetilde{\Psi}^{k}(t)} \big(  x , \Lambda^{k}_{i+1} (t^{k}_{i+1}, x, \lambda) \big) -  v_{\widetilde{\Psi}^{k}(t)} \big(  X^{k}_{i+1} (t, x, \lambda), \Lambda^{k}_{i+1} (t, x, \lambda) \big) \Big | 
\\
&
 \qquad + \Big |  v_{\widetilde{\Psi}^{k}(t)} \big(  X^{k}_{i+1} (t, x, \lambda), \Lambda^{k}_{i+1} (t, x, \lambda) \big) - v_{\Psi^{k}(t)} \big(  X^{k}_{i+1} (t, x, \lambda), \Lambda^{k}_{i+1} (t , x, \lambda) \big) \Big| 
 \\
 &
\vphantom{\Big|} =: I_{1, 1}(x, \lambda) + I_{1, 2}  (x, \lambda) \,.
 \end{split}
\end{equation}
Since $\widetilde{\Psi}^{k} (t) \in \Pp( \B^{Y}_{R})$, hypothesis~$(v_{1})$ implies that
\begin{displaymath}
\begin{split}
I_{1,1} (x, \lambda) & \leq \vphantom{\int} L_{v, R} \big( | X^{k}_{i+1}(t, x, \lambda) - x| + \|  \Lambda^{k}_{i+1} (t^{k}_{i+1}, x, \lambda) -  \Lambda^{k}_{i+1} (t, x, \lambda) \|_{\mathrm{BL}} \big)
\\
&
\leq L_{v, R}\bigg( \int_{t^{k}_{i}}^{t} \big | v_{\widetilde{\Psi}^{k}(\tau)} \big (x, \Lambda^{k}_{i+1}(t^{k}_{i+1}, x, \lambda) \big) \big | \, \di \tau +  \int_{t}^{t^{k}_{i+1}} \big \| \T_{\Psi^{k}_{i}} (  x, \lambda ) \big  \|_{\mathrm{BL}} \, \di \tau \bigg)  \,,
\end{split}
\end{displaymath}
where, in the second inequality, we have used the systems~\eqref{100} and~\eqref{101}. By~$(v_3)$ and~$(\T_{3})$ we can continue with
\begin{equation}\label{e:4}
\begin{split}
I_{1,1} (x, \lambda) & \leq L_{v, R} \bigg( M_{v} \int_{t^{k}_{i}}^{t} \big( 1 + |x| + \| \Lambda^{k}_{i+1}(t^{k}_{i+1}, x, \lambda) \|_{\mathrm{BL}} + m_{1} ( \widetilde{\Psi}^{k} (\tau)) \big) \, \di \tau 
\\
&
\qquad + M_{\T} \int_{t}^{t^{k}_{i+1}} \big( 1 + | x | + \| \lambda\| _{\mathrm{BL}} + m_{1} ( \Psi^{k}_{i} ) \big) \, \di \tau \bigg) 
\\
&
 \vphantom{\int} \leq L_{v, R} (M_{v} + M_{\T} ) ( 1 + 2R) \tau_{k} \,.
\end{split}
\end{equation}

As for $I_{1,2}$, thanks to assumption~$(v_2)$ and to Lemma~\ref{l:Gronwall_bound} we get
\begin{align*}
I_{1,2} & (x, \lambda)  \leq \vphantom{\int} L_{v, R} W_{1} ( \widetilde{\Psi}^{k}(t) , \Psi^{k}(t) )
\\
&
= L_{v, R} \, \sup_{\eta \in \mathrm{Lip}_{1} ( Y)} \bigg\{ \int_{Y} \eta (x' , \lambda' ) \, \di (\widetilde{\Psi}^{k}(t) - \Psi^{k}(t) ) ( x' , \lambda' ) 
\bigg\}
\\
&
= L_{v, R} \, \sup_{\eta \in \mathrm{Lip}_{1} ( Y)} \bigg \{ \int_{Y} \eta (x, \Lambda^{k}_{i+1} ( t^{k}_{i+1}, x' , \lambda' )) - \eta (X^{k}_{i+1} ( t, x' , \lambda' ), \Lambda^{k}_{i+1} (t, x' , \lambda' ) )  \, \di \Psi^{k}_{i} ( x' , \lambda' ) \bigg\}
\\
&
\leq  L_{v, R} \int_{Y} | x  - X^{k}_{i+1} ( t, x' , \lambda' )| + \| \Lambda^{k}_{i+1} ( t^{k}_{i+1}, x' , \lambda' ) ) -  \Lambda^{k}_{i+1} (t, x' , \lambda' ) \|_{\mathrm{BL}}  \, \di \Psi^{k}_{i} ( x' , \lambda' )
\\
&
\leq  L_{v, R} \int_{Y} \bigg( \int_{t^{k}_{i}}^{t} \big | v_{\widetilde{\Psi}^{k}(\tau)} ( x,  \Lambda^{k}_{i+1} ( t^{k}_{i+1}, x' , \lambda' ) )  \big | \, \di \tau  + \int_{t}^{t^{k}_{i+1}} \big \| \T_{\Psi^{k}_{i}} (x' , \lambda' ) \big \|_{\mathrm{BL}} \, \di \tau \bigg) \, \di \Psi^{k}_{i}( x' , \lambda' )
\\
&
\leq  L_{v, R} \, \tau_{k}  \int_{Y}  \Big( \big | v_{\widetilde{\Psi}^{k}(\tau)} ( x, \Lambda^{k}_{i+1} ( t^{k}_{i+1}, x' , \lambda' ) )  \big | + \big \| \T_{\Psi^{k}_{i}}( x' , \lambda' ) \big \|_{\mathrm{BL}} \Big) \, \di \Psi^{k}_{i}( x' , \lambda' ) \,.
\end{align*}
Making use of~$(v_3)$ and~$(\T_3)$ and recalling Lemma~\ref{l:Gronwall_bound} we can continue with
\begin{equation}\label{e:5}
\begin{split}
I_{1,2} (x, \lambda) & \leq L_{v, R} (M_{v} + M_{\T}) \, \tau_{k}\int_{Y}  \Big( 1 + | x' | + \| \Lambda^{k}_{i+1} (t^{k}_{i+1}, x' , \lambda' ) \|_{\mathrm{BL}} + \| \lambda'  \|_{\mathrm{BL}} 
\\
&
\phantom{\leq L_{v, R} (M_{v} + M_{\T}) \, \tau_{k}\int_{Y}  \Big(} + m_{1} (\widetilde{\Psi}^{k} (t) ) + m_{1}(\Psi^{k}_{i}) \Big) \,  \, \di \Psi^{k}_{i} ( x' , \lambda' )
 \\
&
\leq 3 L_{v, R}( M_{v} + M_{\T}) (1+R) \tau_{k}\,. 
\end{split}
\end{equation}
Combining~\eqref{e:2}--\eqref{e:5} we get
\begin{equation}\label{e:6}
I_{1} (x, \lambda) \leq C_{1} \tau_{k}
\end{equation}
for some positive constant~$C_{1}$ independent of~$k$,~$t$,~$\varphi$, and $(x, \lambda) \in \spt \Psi^{k}_{i}$.

Let us now estimate~$I_{2}$.
By Lemma~\ref{l:Gronwall_bound} and by assumption~$(\T_2)$ we get
\begin{equation}\label{e:8}
I_{2} (x, \lambda)  \leq L_{\T, R} \big( | x - X^{k}_{i+1} ( t, x, \lambda)| + \| \lambda - \Lambda^{k}_{i+1} (t, x, \lambda) \|_{\mathrm{BL}} + W_{1} ( \underline{\Psi}^{k} (t) , \Psi^{k}(t))\big) \,.
\end{equation}
Arguing as in~\eqref{e:2}--\eqref{e:6} we deduce from~\eqref{e:8} and from the hypotheses~$(v_{1})$,~$(v_3)$, and~$(\T_{2})$ that 
\begin{equation}\label{e:10}
I_{2}(x, \lambda) \leq C_{2}\tau_{k}
\end{equation}
for some positive constant~$C_{2}$ independent of~$k$,~$t$,~$\varphi$, and $(x, \lambda) \in \spt \Psi^{k}_{i}$.

We are now in a position to conclude the proof of the proposition. We rewrite~\eqref{e:1} as
\begin{align*}
\frac{\di}{\di t} & \int_{Y} \varphi (x, \lambda) \, \di \Psi^{k}(t)(x, \lambda) 
\\
&
= \int_{Y} \nabla_{x} \varphi \big ( X^{k}_{i+1}( t, x, \lambda ) , \Lambda^{k}_{i+1} ( t, x, \lambda) \big) \cdot v_{\Psi^{k} (t)} \big(  X^{k}_{i+1} (t, x, \lambda), \Lambda^{k}_{i+1} (t, x, \lambda) \big) \, \di \Psi^{k}_{i} (x, \lambda)
\\
&
\quad + \int_{Y} \nabla_{\lambda} \varphi  \big ( X^{k}_{i+1}( t, x, \lambda ) , \Lambda^{k}_{i+1} ( t, x, \lambda) \big) \cdot \T_{\Psi^{k}(t)} ( X^{k}_{i+1}( t, x, \lambda), \Lambda^{k} (t, x, \lambda) )  \, \di \Psi^{k}_{i} (x, \lambda)
\\
&
\quad + \int_{Y}  \nabla_{x} \varphi \big ( X^{k}_{i+1}( t, x, \lambda ) , \Lambda^{k}_{i+1} ( t, x, \lambda) \big) \cdot \big( v_{\widetilde{\Psi}^{k}(t)} \big( x , \Lambda^{k}_{i+1} (t^{k}_{i+1}, x, \lambda) \big) 
\\
&
\qquad \phantom{\int_Y} - v_{\Psi^{k}(t)} \big(  X^{k}_{i+1} (t, x, \lambda), \Lambda^{k}_{i+1} (t , x, \lambda) \big) \, \di \Psi^{k}_{i} ( x, \lambda)
\\
&
\quad + \int_{Y} \nabla_{\lambda} \varphi  \big ( X^{k}_{i+1}( t, x, \lambda ) , \Lambda^{k}_{i+1} ( t, x, \lambda) \big) \cdot \big(  \T_{\underline{\Psi}^{k}(t)} ( x, \lambda ) 
\\
&
\qquad \phantom{\int_Y} - \T_{\Psi^{k}(t)} (X^{k}_{i+1} (t, x, \lambda) , \Lambda^{k}(t, x, \lambda) )\big)  \, \di \Psi^{k}_{i} (x, \lambda)   
\\
&
=  \int_{Y} \nabla \varphi (x, \lambda) \cdot b_{\Psi^{k}(t)} (x, \lambda) \, \di \Psi^{k}(t) (x, \lambda) 
\\
&
\quad + \int_{Y}  \nabla_{x} \varphi \big ( X^{k}_{i+1}( t, x, \lambda ) , \Lambda^{k}_{i+1} ( t, x, \lambda) \big) \cdot \big( v_{\widetilde{\Psi}^{k}(t)} \big( x , \Lambda^{k}_{i+1} (t^{k}_{i+1}, x, \lambda) \big) 
\\
&
\qquad \phantom{\int_Y} - v_{\Psi^{k}(t)} \big(  X^{k}_{i+1} (t, x, \lambda), \Lambda^{k}_{i+1} (t , x, \lambda) \big) \, \di \Psi^{k}_{i} ( x, \lambda)
\\
&
\quad + \int_{Y} \nabla_{\lambda} \varphi  \big ( X^{k}_{i+1}( t, x, \lambda ) , \Lambda^{k}_{i+1} ( t, x, \lambda) \big) \cdot \big(  \T_{\underline{\Psi}^{k}(t)} ( x, \lambda ) 
\\
&
\qquad \phantom{\int_Y} - \T_{\Psi^{k}(t)} (X^{k}_{i+1} (t, x, \lambda) , \Lambda^{k}(t, x, \lambda) )\big)  \, \di \Psi^{k}_{i} (x, \lambda)   \,.
 \end{align*}
We conclude by noticing that, thanks to~\eqref{e:6} and~\eqref{e:10}, the last two integrals on the right-hand side of the above equality can be estimated by
\begin{displaymath}
\| \varphi \|_{C^{1}_{b}} \int_{Y} ( I_{1} (x, \lambda)  + I_{2} (x, \lambda) )\, \di \Psi^{k}_{i}(x, \lambda)  \leq C \| \varphi \|_{C^{1}_{b}} \tau_{k}\,,
\end{displaymath}
for a positive constant~$C$ independent of~$k$,~$t$, and~$\varphi$.
\end{proof}

\begin{theorem}\label{t:limit}
Let $\widehat{\Psi} \in \Pp_{c} (Y)$ and let $\Psi^{k}( \cdot)$ be defined as in~\eqref{e:Psi}. Then, $W_{1}(\Psi^{k}(t), \Psi(t)) \to 0$ uniformly in~$[0,T]$, where the curve $\Psi \in C([0,T]; (\Pp_{1}(Y), W_{1}))$ is the unique solution of~\eqref{e:cont_eq} with initial condition~$\Psi(0) = \widehat{\Psi}$.
\end{theorem}

\begin{proof}
The existence and uniqueness of the solution to equation~\eqref{e:cont_eq} follow from~\cite[Theorem~3.5]{MorSol19}, so that $\Psi \in C([0,T]; ( \Pp_{1}(Y), W_{1}))$ is well defined.

Let $\phi \in C^{1}_{b} ([0,T] \times \overline{Y})$. In view of Proposition~\ref{p:discrete_equation}, for every~$k\in\N$,~$i\in\{0,\ldots,k-1\}$, and every $t \in (t^{k}_{i}, t^{k}_{i+1})$, we have
\begin{displaymath}
\begin{split}
\frac{\di}{\di t} \int_{Y} \phi (t, x, \lambda) \, \di \Psi^{k}(t) (x, \lambda) = & \ \int_{Y} \partial_{t} \phi(t, x, \lambda) \, \di \Psi^{k}(t) (x, \lambda)
\\
&
 + \int_{Y} \nabla\phi (t, x, \lambda) \cdot b_{\Psi^{k}(t)} (x, \lambda) \, \di \Psi^{k}(t) (x, \lambda) + \theta_{k} ( \phi(t, \cdot, \cdot)) \,,
\end{split}
\end{displaymath}
where $|\theta_{k} ( \phi(t, \cdot, \cdot))| \leq C \tau_{k} \| \phi \| _{C^{1}_{b} ( [0,T]\times \overline{Y})}$ uniformly in~$[0,T]$. By integrating the previous equality over time, we deduce that
\begin{equation}\label{e:11}
\begin{split}
 \int_{Y} \phi (t, x, \lambda) & \, \di \Psi^{k}(t) (x, \lambda)   - \int_{Y} \phi (0, x, \lambda) \, \di \widehat{\Psi} (x, \lambda) = \int_{0}^{t} \int_{Y} \partial_{t} \phi(\tau, x, \lambda) \, \di \Psi^{k}(\tau) (x, \lambda) \, \di \tau 
\\
&
 +\int_{0}^{t} \int_{Y} \nabla\phi (\tau, x, \lambda) \cdot b_{\Psi^{k}(\tau)} (x, \lambda) \, \di \Psi^{k}(\tau) (x, \lambda) \, \di \tau  + \int_{0}^{t} \theta_{k} ( \phi(\tau, \cdot, \cdot))\, \di \tau  \,.
\end{split}
\end{equation}

In order to pass to the limit in~\eqref{e:11}, we have to determine a candidate limit for~$\Psi^{k}(t)$. In Lemma~\ref{l:Gronwall_bound} we have already shown that the supports of~$\Psi^{k}(t)$ are contained in a compact subset of~$\overline{Y}$. We now show the equicontinuity of the sequence~$\Psi^{k}$ with respect to~$W_{1}$. Given $s, t \in [0,T]$, we show that $W_{1} ( \Psi^{k}(s), \Psi^{k}(t)) \leq L | s - t|$ for some $L>0$ independent of~$k$. By triangle inequality, it is enough to show it for~$s, t \in [t^{k}_{i}, t^{k}_{i+1}]$. Arguing as in the proof of Proposition~\ref{p:discrete_equation} we obtain
\begin{equation}
W_{1}( \Psi^{k}(s) , \Psi^{k}(t)) \leq 3 ( M_{v} + M_{\T}) (1+ R) |s-t|\,,
\end{equation}
where~$R$ has been defined in Lemma~\ref{l:Gronwall_bound}. Hence, Ascoli-Arzel\`a theorem yields that there exists $\overline{\Psi} \in C([0,T]; ( \Pp_{1} ( Y) , W_{1}))$ such that, up to a subsequence, $W_{1} ( \Psi^{k}(t), \overline{\Psi} (t)) \to 0 $ uniformly with respect to~$t \in [0,T]$. In particular,~$\overline{\Psi}(0) = \widehat{\Psi}$ and~$\overline{\Psi}(t) \in \Pp( \B^{Y}_{R})$, since $\spt \Psi^{k}(t) \subseteq \B^{Y}_{R}$ for every $k$ and every~$t$.

It remains to show that $\overline{\Psi}$ is a solution to~\eqref{e:cont_eq}, from which we would deduce that $\overline{\Psi} = \Psi$ and that the whole sequence $\Psi^{k}$ converges to~$\Psi$. The first line of~\eqref{e:11} passes to limit as~$k\to \infty$, since the test function~$\phi$ belongs to $C^{1}_{b}([0,T]\times\overline{Y})$ and the convergence of~$\Psi^{k}$ in~$W_{1}$ is uniform in time and implies the narrow convergence. The last term on the right-hand side of~\eqref{e:11} tends to~$0$, since it holds
\begin{displaymath}
\int_{0}^{t} |\theta_{k} (\phi(\tau, \cdot, \cdot)) | \, \di \tau \leq C T  \tau_{k} \| \phi \|_{C^{1}_{b} ( [0,T]\times \overline{Y})}\,.
\end{displaymath}
We conclude by estimating
\begin{equation}\label{e:12}
\begin{split}
\bigg| \int_{0}^{t}  & \!\! \int_{\overline{Y}} \!\! \nabla \phi (\tau, x, \lambda) \cdot b_{ \Psi^{k} (\tau) } (x, \lambda)  \di \Psi^{k} (\tau) (x, \lambda )  \di \tau  - \! \int_{0}^{t}  \!\! \int_{\overline{Y}}\!\! \nabla \phi ( \tau, x, \lambda ) \cdot b_{ \overline{\Psi} (\tau) } (x, \lambda)  \di \overline{\Psi} (\tau) (x, \lambda)  \di \tau  \bigg|
\\
&
\leq \| \phi \|_{C^{1}_{b}} \int_{0}^{t} \int_{\overline Y} \| b_{\Psi^{k}(\tau)} ( x,\lambda) - b_{\overline{\Psi}(\tau)} ( x,\lambda) \|_{\overline{Y}} \, \di \Psi^{k}(\tau) (x, \lambda) \, \di \tau 
\\
&
\qquad + \int_{0}^{t}\bigg| \int_{\overline{Y}}  \nabla \phi( \tau, x, \lambda) \cdot b_{\overline{\Psi}(\tau)} ( x, \lambda) \, \di ( \Psi^{k} ( \tau) - \overline{\Psi}(\tau)) ( x, \lambda) \bigg| \, \di \tau  \,.
\end{split}
\end{equation}
By~\cite[Proposition~3.2]{MorSol19},  Lemma~\ref{l:Gronwall_bound}, and Assumptions~$(v_2)$ and~$(\T_2)$, the first term on the right-hand side of~\eqref{e:12} can be estimated by
\begin{displaymath}
\| \phi \|_{C^{1}_{b}} ( L_{R,v} + L_{R, \T}) \int_{0}^{t} W_{1} ( \Psi^{k}(\tau), \overline{\Psi}(\tau)) \, \di \tau  \to 0 \qquad \text{as $k\to \infty$ uniformly with respect to~$t \in [0,T]$}\,.
\end{displaymath}

As for the second term, we first notice that, by~\cite[Proposition~3.2]{MorSol19} and Lemma~\ref{l:Gronwall_bound}, the function $(x, \lambda) \mapsto b_{\overline{\Psi}(\tau)} ( x, \lambda)$ is continuous from~$\overline{Y}$ to~$\overline{Y}$ and is bounded on~$\B^{Y}_{R}$. Since $\Psi^{k}(\tau)$ converges narrowly to~$\overline{\Psi}(\tau)$, for $\tau \in [0,t]$ we have
\begin{displaymath}
\lim_{k\to\infty} \bigg| \int_{\overline{Y}}  \nabla \phi( \tau, x, \lambda) \cdot b_{\overline{\Psi}(\tau)} ( x, \lambda) \, \di ( \Psi^{k} ( \tau) - \overline{\Psi}(\tau)) ( x, \lambda) \bigg| = 0\,.
\end{displaymath}
Furthermore, by~$(v_3)$ and~$(\T_1)$ we have the uniform bound
\begin{align*}
\bigg| \int_{\overline{Y}}  \nabla & \phi( \tau, x, \lambda) \cdot b_{\overline{\Psi}(\tau)} ( x, \lambda) \, \di ( \Psi^{k} ( \tau) - \overline{\Psi}(\tau)) ( x, \lambda) \bigg| 
\\
&
 \leq 4 \| \phi\|_{C^{1}_{b}} (M_{v} + M_{\T})  ( 1 + R) \| \Psi^{k}(\tau) - \overline{\Psi}(\tau) \|_{\mathrm{TV}} \leq  3 \| \phi\|_{C^{1}_{b}} (M_{v} + M_{\T}) ( 1 + R)
\end{align*}
for $\tau \in [0,t]$. Thus, by dominated convergence also the second term on the right-hand side of~\eqref{e:12} tends to zero as $k\to \infty$.

Eventually, we infer that passing to the limit~$k\to \infty$ in~\eqref{e:11} we get the equality
\begin{displaymath}
\begin{split}
 \int_{Y} \phi (t, x, \lambda) & \, \di \overline{\Psi}(t) (x, \lambda)  - \int_{Y} \phi (0, x, \lambda) \, \di \widehat{\Psi} (x, \lambda) 
 \\
 &= \int_{0}^{t} \int_{Y} \partial_{t} \phi(\tau, x, \lambda) \, \di \overline{\Psi} (\tau) (x, \lambda)
 +\int_{0}^{t} \int_{Y} \nabla\phi (\tau, x, \lambda) \cdot b_{\overline{\Psi}(\tau)} (x, \lambda) \, \di \overline{\Psi} ( \tau) (x, \lambda) 
\end{split}
\end{displaymath}
for every $\phi \in C^{1}_{b} ([0,T]\times \overline{Y})$ and every $t \in [0,T]$. This concludes the proof of the theorem.
\end{proof}

\section{Inhomogeneous replicator dynamics}\label{s:alternative}
We discuss in this section a different discrete-time approximation of the continuity equation~\eqref{e:cont_eq} for the operator~$\T_{\Psi}\colon Y \to \F(U)$ corresponding to the transition operators considered in~\cite{AmbForMorSav18} (see also~\cite[Section~5]{MorSol19}) for the replicator equation, namely
\begin{equation}\label{e:new_T}
\begin{split}
\T_{\Psi} ( x, \lambda ) \coloneqq  \biggl( & \int_{\overline{Y}} \int_{U} J(x, u, x', u') \, \di \lambda' (u') \, \di \Psi ( x', \lambda' )
\\
&
- \int_{U} \int_{\overline{Y}} \int_{U} J(x, u, x', u') \, \di \lambda' (u') \, \di \Psi ( x', \lambda' ) \, \di \lambda (u) \biggr) \lambda
\end{split}
\end{equation} 
defined for every $\Psi \in \Pp_{1}(Y)$ and every $y = (x, \lambda) \in Y$. In~\eqref{e:new_T} we consider a function~$J\colon (\R^{d} \times U)^{2} \to \R$ such that
\begin{itemize}
\item[$(J_1)$] $J$ is locally Lipschitz continuous with respect to all of its variables;

\item[$(J_2)$] there exists $M_{J}>0$ such that for every $(x, u, x', u') \in (\R^{d} \times U)^{2}$
\begin{displaymath}
| J(x, u, x', u') | \leq M_{J} ( 1 + | x| + | x' |) \,.
\end{displaymath}
\end{itemize}
For simplicity of notation, from now on we will write
\begin{align*}
& (J * \Psi) (x, u) \coloneqq \int_{\overline{Y}} \int_{U} J(x, u, x', u') \, \di \lambda' (u') \, \di \Psi ( x', \lambda' )\,,\\[1mm]
& \left\langle J * \Psi , \lambda \right \rangle (x)  \coloneqq \int_{U} (J * \Psi) (x, u) \, \di \lambda (u)\,,
\end{align*}
so that \eqref{e:new_T} can be written as
\begin{equation}\label{e:new_T-a}
\T_{\Psi} (x, \lambda) =  \big((J * \Psi) (x, \cdot) -  \left\langle J * \Psi , \lambda \right \rangle (x) \big) \lambda\,.
\end{equation}

The following proposition holds.

\begin{proposition}{\cite[Proposition~5.8]{MorSol19}}
Under the assumptions~$(J_1)$--$(J_2)$, the operator~$\T_{\Psi}$ defined in~\eqref{e:new_T} satisfies the conditions~$(\T_0)$--$(\T_3)$.
\end{proposition}

We now introduce the \emph{spherical Hellinger distance} between probability measures
\begin{align*}
\HS^{2} ( \lambda_{1}, \lambda_{2}) \coloneqq  \inf\, \biggl\{ \frac{1}{4} \int_{0}^{1} | w_{t}(u) |^{2} \, \di \rho_{t}(u)\, \di t : \, &  \rho \in C([0,1]; \Pp(U)), \\
& \dot{\rho}_t = \bigg( w_t - \int_{U} w_{t} \, \di \rho_{t} \bigg) \rho_{t},\, \rho_{0} = \lambda_{1}, \, \rho_{1} = \lambda_{2}\biggl\}\,,
\end{align*}
defined for every $\lambda_{1}, \lambda_{2} \in \Pp(U)$. For later use, we also define the \emph{Hellinger distance} between nonnegative measures: for every $\mu_{1}, \mu_{2} \in \mathcal{M}_{+}(U)$, we set
\begin{align*}
\He^{2} (\mu_{1}, \mu_{2})&  \coloneqq  \inf\, \biggl\{ \frac{1}{4} \! \int_{0}^{1} | w_{t}(u) |^{2} \, \di \rho_{t} (u)\, \di t :   \rho \in C([0,1]; \mathcal{M}_{+}(U)), 
\, \dot{\rho}_t =  w_t \, \rho_{t},\, \rho_{0} = \mu_{1}, \, \rho_{1} = \mu_{2} \biggl\}
\\
& = \int_{U} \bigg[ \bigg( \frac{\di \mu_{1}}{\di \mu^*} \bigg)^{\frac12} - \bigg( \frac{\di \mu_{2}}{\di \mu^*} \bigg)^{\frac12} \bigg]^{2} \, \di \mu^* .
\end{align*}
where $\mu^{*} \in \mathcal{M}_{+}(U)$ is such that $\mu_{1}, \, \mu_{2} \ll \mu^{*}$. We notice that~$\HS^{2}$ can be expressed in terms of~$\He^{2}$ through
\begin{equation}\label{e:arccos}
\HS^{2}(\lambda_{1}, \lambda_{2}) = \arccos \bigg( 1 - \frac{\He^{2}(\lambda_{1}, \lambda_{2})^{2} }{2}\bigg) \qquad \text{for every $\lambda_{1}, \lambda_{2} \in \Pp(U)$\,,}
\end{equation}
and that the following chain of inequalities holds:
\begin{equation}\label{e:equivalence}
\| \lambda_{1} - \lambda_{2} \|_{\mathrm{BL}} \leq \| \lambda_{1} - \lambda_{2} \|_{\mathrm{TV}} \leq 2 \,\He (\lambda_{1}, \lambda_{2}) \leq 2 \,\HS( \lambda_{1}, \lambda_{2}) \qquad \text{for every $\lambda_{1}, \lambda_{2} \in \Pp(U)$}\,.
\end{equation}

In the spatially homogeneous case, the replicator equation is a \emph{generalized minimizing movement}~\cite{AmbGigSav08} for the functional
\begin{displaymath}
\J_{\mathrm{hom}} (\lambda) \coloneqq - \frac{1}{8} \int_{U}\int_{U} J(u, u') \, \di \lambda(u) \, \di \lambda(u') 
\end{displaymath}
with respect to the spherical Hellinger distance. In the spatially inhomogeneous setting, the payoff functional has a bilinear dependence on $\Psi$ and $\lambda$: for every $\Psi \in \Pp_{1}(Y)$ and every $(x, \lambda) \in Y$ we set 
\begin{equation}\label{e:J}
\J_{\Psi} (x, \lambda) \coloneqq - \frac{1}{4} \left\langle (J * \Psi) ,  \lambda \right\rangle ,
\end{equation}
(the factor $\frac14$ instead of $\frac18$ is due to the dependence on $\lambda$ which is now linear).  We modify the scheme in Section \ref{s:scheme} by replacing the finite difference~\eqref{e:ODE1} with a minimizing movement. Namely, in the interval $[t^{k}_{i}, t^{k}_{i+1})$ let $\Psi^{k}_{i} \in \Pp (U)$ be given and define, for every $( \hat{x}, \hat{\lambda}) \in Y$,
\begin{equation}\label{e:HS_step}
\lambda_{(\hat{x}, \hat{\lambda}), i+1} \coloneqq \argmin \, \bigg\{ \J_{\Psi^{k}_{i}} (\hat{x} , \lambda) + \frac{1}{2\tau_{k}} \HS^{2}(\lambda, \hat{\lambda}) : \, \lambda \in \Pp(U) \bigg\}\,.
\end{equation}
Notice that the measure $\lambda_{(\hat{x}, \hat{\lambda}), i+1} \in \Pp(U)$ is well-defined, as~$\Pp(U)$ is compact and the functional in~\eqref{e:HS_step} is strictly convex. Therefore, we can define~$\lambda^{k}_{(\hat{x}, \hat{\lambda}), i+1}$, $\Lambda^{k}_{i+1}$, and~$\tilde{\Psi}^{k}_{i+1}$ exactly as in~\eqref{100},~\eqref{e:lambda}, and~\eqref{e:tilde_Psi}, respectively. The second step~\eqref{e:ODE2} in the space variable remains instead the same, so that~$x^{k}_{(\hat{x}, \hat{\lambda}), i+1}$,~$X^{k}_{i+1}$,~$\Psi^{k}_{i+1}$ are as in~\eqref{101}, \eqref{e:X}, and \eqref{e:Psi}, respectively. We further refer to~\eqref{e:Psi}, \eqref{e:Psi_tilde}, and \eqref{e:underline_Psi} for the definition of the interpolation curves~$\Psi^{k}$,~$\widetilde{\Psi}^{k}$, and~$\underline{\Psi}^{k}$.

The next lemma gives an estimate on the size of the support of~$\Psi^{k}_{i+1}$ and~$\widetilde{\Psi}^{k}_{i+1}$, showing that they belong to~$\Pp_{c}(Y) \subseteq \Pp_{1}(Y)$ for every~$k\in \N$ and every~$i\in\{0,\ldots,k-1\}$.

\begin{lemma}\label{l:Gronwall2}
Let $\widehat{\Psi} \in \Pp_{c}(Y)$. Then, there exists $R>0$ such that, for every $k \in \mathbb{N}$ and every $t\in [0,T]$, $\Psi^{k}(t), \, \underline{\Psi}^{k}(t), \, \widetilde{\Psi}^{k}(t) \in \Pp(\B^{Y}_{R})$.
\end{lemma}

\begin{proof}
Let us define the piecewise constant interpolation functions~$\underline{\lambda}^{k}$,~$\overline{\lambda}^{k}$, and~$\underline{x}^{k}$ as in~\eqref{e:110}, and let~$x^{k}$ and~$\lambda^{k}$ be the corresponding piecewise affine interpolations. Then, by~\eqref{e:HS_step} we have that $\underline{\lambda}^{k}(t, x_{0}, \lambda_{0}), \, \overline{\lambda}^{k}(t, x_{0}, \lambda_{0})   \in \Pp(U)$ for every $t \in [0,T]$ and every $(x_{0}, \lambda_{0}) \in \spt \widehat{\Psi}$, so that 
\begin{displaymath}
\| \underline{\lambda}^{k} (t, x_{0}, \lambda_{0}) \|_{\mathrm{BL}}, \| \overline{\lambda}^{k} (t, x_{0}, \lambda_{0}) \|_{\mathrm{BL}} \leq 1\,.
\end{displaymath}
Following step by step the proof of~\eqref{e:104} and~\eqref{e:107}, we also deduce that there exists~$R>0$ such that
\begin{equation}\label{e:111}
| x_{k}(t, x_{0}, \lambda_{0}) | \leq R \qquad \text{for every $t \in [0,T]$, every $(x_{0}, \lambda_{0}) \in \spt\, \widehat{\Psi}$, and every $k\in \mathbb{N}$}\,.
\end{equation}
We notice that, being the step~\eqref{e:HS_step} defined through a minimization in~$\Pp(U)$ and not through a finite difference, the estimate~\eqref{e:111} holds for every~$k$, and not only for~$k$ large. Moreover,~\eqref{e:111} yields that $\Psi^{k}(t)$, $\underline{\Psi}^{k}(t)$, $\widetilde{\Psi}^{k}(t) \in \Pp(\B^{Y}_{R})$. 
\end{proof}

In order to write the equivalent of Proposition~\ref{p:discrete_equation}, we have to determine an approximate Euler-Lagrange equation for the minimization problem~\eqref{e:HS_step}. This is the content of the following proposition, written here for generic $\Psi$,~$x$, and~$\lambda$.

\begin{proposition}\label{p:HS_euler}
Let $R>0$. Assume that $\Psi \in \Pp (\B^{Y}_{R})$, $(x, \lambda) \in \B^{Y}_{R}$, and let $\tilde{\lambda} \in \Pp(U)$ be the solution to
\begin{equation}\label{e:112}
\min \, \bigg\{\J_{\Psi} (x , \rho) + \frac{1}{2\tau_{k}} \HS^{2}(\rho, \lambda) : \, \rho \in \Pp(U) \bigg\} \,.
\end{equation}
Then, there exists a constant $C= C(R)>0$ such that
\begin{align}
& \HS ( \tilde{\lambda}, \lambda) \leq C \tau_{k} \,, \label{e:Lipschitz} \\[1mm]
& \bigg\| \frac{\tilde{\lambda} - \lambda}{ \tau_{k}} - \T_{\Psi}(x, \tilde{\lambda})  \bigg\|_{\mathrm{BL}} \leq C \tau_{k}( 1 + \tau_{k}) \,. \label{e:approx_EL}
\end{align}
\end{proposition}

\begin{proof}
Inequality~\eqref{e:Lipschitz} follows from the minimality of~$\tilde{\lambda}$. Indeed, we have that
\begin{equation}\label{e:113}
 \frac{1}{2\tau_{k}} \HS^{2}(\tilde{\lambda}, \lambda) \leq \big| \J_{\Psi} (x , \lambda) - \J_{\Psi} (x , \tilde{\lambda}) \big| \,. 
\end{equation}
By definition~\eqref{e:J} of~$\J_{\Psi}$, by~$(J_1)$, by the assumptions $\Psi \in \Pp (\B^{Y}_{R})$, $(x, \lambda) \in \B^{Y}_{R}$, and by~\eqref{e:equivalence}, we continue in~\eqref{e:113} with
\begin{equation}\label{e:114}
 \frac{1}{2\tau_{k}} \HS^{2}(\tilde{\lambda}, \lambda) \leq \frac{ M_{J} }{2} (1 + R) \| \tilde{\lambda} - \lambda \|_{\mathrm{TV}}  \leq  M_J (1+R) \HS( \tilde{\lambda}, \lambda) \,.
\end{equation}
From~\eqref{e:114} we deduce~\eqref{e:Lipschitz}.

In order to prove~\eqref{e:approx_EL}, we write explicitly the Euler-Lagrange equation of~\eqref{e:112}. Here, we follow the lines of~\cite[Section~4]{GalMon17}. For every $\varphi \in \mathrm{Lip}(U)$ with~$\| \varphi \|_{\mathrm{Lip}} \leq 1$, we consider the auxiliary system
\begin{equation}\label{e:115}
\left\{ \begin{array}{ll}
\partial_{\varepsilon} \lambda_{\varepsilon} = ( \varphi - \left\langle \varphi, \lambda_{\varepsilon} \right \rangle ) \lambda_{\varepsilon}\,,\\[1mm]
\lambda_{0} = \tilde{\lambda}\,.
\end{array}\right.
\end{equation}
In view of~\cite[Section~I.3, Theorem~1.4, Corollary~1.1]{MR0348562}, the ODE system~\eqref{e:115} admits a unique solution $\lambda^{\varphi}_{\varepsilon} \in \Pp(U)$ for $\varepsilon > 0$. Moreover, if~$\tilde{\lambda} \ll \mu$, it is easy to check that $\lambda^{\varphi}_{\varepsilon} \ll \mu$ for $\varepsilon>0$. In the sequel, we fix $\mu^{*} \in \Pp(U)$ such that $\lambda, \tilde{\lambda} \ll \mu^*$.

Given $\lambda^{\varphi}_{\varepsilon}$, the Euler-Lagrange equation of~\eqref{e:112} reads
\begin{equation}\label{e:116}
\frac{\di}{\di \varepsilon}\bigg|_{\varepsilon = 0}  \!\!\! \J_{\Psi} ( x, \lambda^{\varphi}_{\varepsilon} ) + \frac{1}{2 \tau_{k}} \frac{\di}{\di \varepsilon}\bigg|_{\varepsilon = 0} \!\!\! \HS^{2} ( \lambda^{\varphi}_{\varepsilon}, \lambda) = 0 \,.
\end{equation}
We compute the two derivatives appearing in~\eqref{e:116} separately. In view of~\eqref{e:115}, we have that
\begin{equation}\label{e:117}
\begin{split}
\frac{\di}{\di \varepsilon}\bigg|_{\varepsilon = 0} \J_{\Psi} ( x, \lambda^{\varphi}_{\varepsilon} ) & = -  \frac{1}{4}\frac{\di}{\di \varepsilon}\bigg|_{\varepsilon = 0} \langle (J * \Psi) ,  \lambda_{\varepsilon}^{\varphi} \rangle =  -  \frac{1}{4}\big\langle (J * \Psi) ,  \big( \varphi - \big\langle \varphi , \tilde{\lambda} \big\rangle \big) \tilde{\lambda} \big\rangle 
\\
&
= - \frac{1}{4} \big\langle \big( (J * \Psi) - \big\langle (J * \Psi) , \tilde{\lambda} \big\rangle \big) \tilde{\lambda} , \varphi \big\rangle = - \frac{1}{4} \big\langle\T_{\Psi} ( x, \tilde{\lambda}) , \varphi \big\rangle\,,
\end{split}
\end{equation}
where, in the last equality, we have used~\eqref{e:new_T-a}.

To compute the second term on the left-hand side of~\eqref{e:116}, we first notice that, since $\tilde\lambda , \lambda, \lambda^{\varphi}_{\varepsilon} \ll \mu^{*}$, we can write
\begin{displaymath}
\HS^{2}(  \lambda^{\varphi}_{\varepsilon}, \lambda ) = \arccos \left( 1 - \frac{\He^{2}(\lambda^{\varphi}_{\varepsilon}, \lambda)^2}{2} \right), \qquad  \He^{2}(\lambda^{\varphi}_{\varepsilon}, \lambda) = \int_{U} \left[ \bigg( \frac{\di \lambda^{\varphi}_{\varepsilon}}{\di \mu^{*}} \bigg)^{\frac12} - \bigg( \frac{\di \lambda}{\di \mu^{*}}\bigg)^{\frac12} \right]^{2}   \di \mu^{*} .
\end{displaymath}
Defining~$\delta_{k}(\tilde{\lambda}, \lambda) \in [0,1]$ such that $1 - \delta_{k} ( \tilde{\lambda}, \lambda) = \frac{1}{\sqrt { 1 - \frac{\He^{2} ( \tilde{\lambda}, \lambda)^{2}}{4}}}$,
we have that
\begin{equation}\label{e:119}
\begin{split}
\!\!\! \frac{\di}{\di \varepsilon}\bigg|_{\varepsilon = 0}\!\!\!  \HS^{2} ( \lambda^{\varphi}_{\varepsilon}, \lambda) & = \big( 1 - \delta _{k}( \tilde{\lambda}, \lambda) \big) \frac{\di}{\di \varepsilon}\bigg|_{\varepsilon = 0} \!\!\! \He^{2} ( \lambda^{\varphi}_{\varepsilon}, \lambda) 
\\
&
= 2\big( 1 - \delta_{k} ( \tilde{\lambda}, \lambda) \big) \! \int_{U} \bigg[  \bigg( \frac{\di \tilde \lambda}{\di \mu^{*}} \bigg)^{\frac12} - \bigg( \frac{\di \lambda}{\di \mu^{*}}\bigg)^{\frac12} \bigg] \frac{\di}{\di \varepsilon}\bigg|_{\varepsilon = 0} \bigg( \frac{\di \lambda^{\varphi}_{\varepsilon}}{\di \mu^{*}} \bigg)^{\frac12} \, \di \mu^{*} 
\\
&
=\big( 1 - \delta_{k} ( \tilde{\lambda}, \lambda) \big) \! \int_{U} \bigg[  \bigg( \frac{\di \tilde \lambda}{\di \mu^{*}} \bigg)^{\frac12} - \bigg( \frac{\di \lambda}{\di \mu^{*}}\bigg)^{\frac12} \bigg]  \bigg( \frac{\di \tilde \lambda}{\di \mu^{*}} \bigg)^{-\frac12} \big( \varphi - \langle \varphi, \tilde{\lambda} \rangle \big) \, \di  \tilde{\lambda} 
\\
&
=\big( 1 - \delta_{k} ( \tilde{\lambda}, \lambda) \big) \bigg\langle  \big( \varphi - \langle \varphi, \tilde{\lambda} \rangle \big)  \mu^{*} , \bigg[  \bigg( \frac{\di \tilde \lambda}{\di \mu^{*}} \bigg)^{\frac12} - \bigg( \frac{\di \lambda}{\di \mu^{*}}\bigg)^{\frac12} \bigg] \bigg( \frac{\di \tilde \lambda}{\di \mu^{*}} \bigg)^{\frac12}  \bigg \rangle \,.
\end{split}
\end{equation}
Using the algebraic equality $2 (a - b) a = a^{2} - b^{2} + (a - b )^{2}$, we continue in~\eqref{e:119} with
\begin{equation}\label{e:119a}
\begin{split}
\frac{\di}{\di \varepsilon}\bigg|_{\varepsilon = 0}  \!\!\! \HS^{2} ( \lambda^{\varphi}_{\varepsilon}, \lambda)  =&
 \frac{\big( 1 - \delta_{k} ( \tilde{\lambda}, \lambda) \big)}{2} \bigg\langle \big( \varphi - \langle \varphi, \tilde{\lambda} \rangle \big)  \mu^{*} , \bigg(   \frac{\di \tilde \lambda}{\di \mu^{*}}  -  \frac{\di \lambda}{\di \mu^{*}} \bigg) \bigg \rangle 
\\
& + \frac{\big( 1 - \delta_{k} ( \tilde{\lambda}, \lambda) \big)}{2} \bigg\langle \big( \varphi - \langle \varphi, \tilde{\lambda} \rangle \big)  \mu^{*} ,  \bigg[  \bigg(\frac{\di \tilde \lambda}{\di \mu^{*}}\bigg)^{\frac12}  -  \bigg(\frac{\di \lambda}{\di \mu^{*}}\bigg)^{\frac12} \bigg]^{2} \bigg \rangle 
\\
=& \frac{\big( 1 - \delta_{k} ( \tilde{\lambda}, \lambda) \big)}{2} \bigg\langle  \tilde{\lambda} - \lambda , \big( \varphi - \langle \varphi, \tilde{\lambda} \rangle \big)  \bigg \rangle 
\\
& + \frac{\big( 1 - \delta_{k} ( \tilde{\lambda}, \lambda) \big)}{2} \bigg\langle \big( \varphi - \langle \varphi, \tilde{\lambda} \rangle \big)  \mu^{*} , \bigg[  \bigg(\frac{\di \tilde \lambda}{\di \mu^{*}}\bigg)^{\frac12}  -  \bigg(\frac{\di \lambda}{\di \mu^{*}}\bigg)^{\frac12} \bigg]^{2} \bigg \rangle 
\\
=& \frac{\big( 1 - \delta_{k} ( \tilde{\lambda}, \lambda) \big)}{2} \langle  \tilde{\lambda} - \lambda , \varphi   \rangle 
\\
& + \frac{\big( 1 - \delta_{k} ( \tilde{\lambda}, \lambda) \big)}{2} \bigg\langle \big( \varphi - \langle \varphi, \tilde{\lambda} \rangle \big)  \mu^{*} ,  \bigg[  \bigg(\frac{\di \tilde \lambda}{\di \mu^{*}}\bigg)^{\frac12}  -  \bigg(\frac{\di \lambda}{\di \mu^{*}}\bigg)^{\frac12} \bigg]^{2} \bigg \rangle \,, 
\end{split}
\end{equation}
where, in the last equality, we have used the fact that $\tilde{\lambda}, \lambda \in \Pp(U)$.

In order to conclude with~\eqref{e:approx_EL}, we estimate~$\delta_{k} (\tilde{\lambda}, \lambda)$ and the last term on the right-hand side of~\eqref{e:119a}. In view of~\eqref{e:arccos},~\eqref{e:equivalence}, and~\eqref{e:Lipschitz}, it is easy to check that
\begin{equation}\label{e:delta2}
\delta_{k} ( \tilde{\lambda}, \lambda ) \leq c \tau_{k}^{2}
\end{equation}
for some positive constant~$c = c(R) >0$. Since~$\| \varphi \|_{\mathrm{Lip}} \leq 1$ and~\eqref{e:114} holds, we have that
\begin{equation}\label{e:120}
\begin{split}
\!\!\!\!\! \bigg| \frac{(   1 - \delta_{k} ( \tilde{\lambda}, \lambda) )}{2} \bigg \langle \! \big( \varphi - \langle \varphi, \tilde{\lambda} \rangle \big)  \mu^{*} ,  \bigg[  \bigg(\frac{\di \tilde \lambda}{\di \mu^{*}}\bigg)^{\frac12}  \!\! -  \bigg(\frac{\di \lambda}{\di \mu^{*}}\bigg)^{\frac12} \bigg]^{2} \bigg \rangle \bigg|  \leq \He^{2}( \tilde{\lambda}, \lambda) \leq 4 M_{J}^{2} ( 1 + R)^{2} \tau_{k}^{2} \,.
\end{split}
\end{equation}

Combining~\eqref{e:equivalence},~\eqref{e:Lipschitz}, and~\eqref{e:116}--\eqref{e:120}, we deduce that
\begin{displaymath}
\begin{split}
& \bigg\| \frac{\tilde{\lambda} - \lambda}{\tau_{k}} - \T_{\Psi}(x, \tilde{\lambda}) \bigg\|_{\mathrm{BL}} 
\leq \delta_{k} ( \tilde{\lambda}, \lambda)  \bigg\|  \frac{\tilde{\lambda} - \lambda}{ \tau_{k}} \bigg\|_{\mathrm{BL}} + 8 M_{J}^{2} ( 1 + R)^{2} \tau_{k} \leq C \tau_{k} ( 1 + \tau_{k})\,, 
\end{split}
\end{displaymath}
for some positive constant~$C = C(R)$. This concludes the proof of the proposition.
\end{proof}

We are now in a position to state the equivalent of Proposition~\ref{p:discrete_equation}.

\begin{proposition}\label{p:discrete_equation2}
There exists~$C>0$ such that for every $\varphi \in C_{b}^{1}( \R^{d} \times \F(U))$, every~$k \in \mathbb{N}$, every $i \in \{0, \ldots, k-1\}$, and every $t \in (t^{k}_{i}, t^{k}_{i+1})$,
\begin{equation}\label{e:approx_equation2}
\frac{\di}{\di t} \int_{Y} \varphi (x, \lambda) \, \di \Psi^{k}(t)(x, \lambda) = \int_{Y} \nabla \varphi (x, \lambda) \cdot b_{\Psi^{k}(t)} (x, \lambda) \, \di \Psi^{k}(t) (x, \lambda) + \vartheta_{k}(\varphi) \,,
\end{equation}
where $|\vartheta_{k}(\varphi)|  \leq C \| \varphi \|_{C^{1}_{b}} \tau_{k}$. 

\end{proposition}

\begin{proof}
Along the proof we denote by~$C$ a generic positive constant independent of~$i$,~$k$,~$t$, and~$\varphi$, that may vary from line to line.

We follow step by step the proof of Proposition~\ref{p:discrete_equation}. For every test function $\varphi \in C_{b}^{1}(\R^{d} \times \F(U))$ and every $t \in (t^{k}_{i}, t^{k}_{i+1})$, by definition of~$\Psi^{k}(t)$ we have that
\begin{equation}\label{e:121}
\begin{split}
\frac{\di}{\di t} & \int_{Y} \varphi (x, \lambda) \, \di \Psi^{k}(t)(x, \lambda) = \frac{\di}{\di t} \int_{Y} \varphi \big ( X^{k}_{i+1}( t, x, \lambda ) , \Lambda^{k}_{i+1} ( t, x, \lambda) \big) \, \di \Psi^{k}_{i} (x, \lambda) 
\\
&
= \int_{Y} \nabla_{x} \varphi \big ( X^{k}_{i+1}( t, x, \lambda ) , \Lambda^{k}_{i+1} ( t, x, \lambda) \big) \cdot v_{\widetilde{\Psi}^{k} (t)} \big(  x , \Lambda^{k}_{i+1} (t^{k}_{i+1}, x, \lambda) \big) \, \di \Psi^{k}_{i} (x, \lambda)
\\
&
\quad + \int_{Y} \nabla_{\lambda} \varphi  \big ( X^{k}_{i+1}( t, x, \lambda ) , \Lambda^{k}_{i+1} ( t, x, \lambda) \big)\cdot\dot{\Lambda}^{k}_{i+1}(t, x, \lambda) \, \di \Psi^{k}_{i} (x, \lambda)
\\
&
=  \int_{Y} \nabla_{x} \varphi \big ( X^{k}_{i+1}( t, x, \lambda ) , \Lambda^{k}_{i+1} ( t, x, \lambda) \big) \cdot v_{\widetilde{\Psi}^{k} (t)} \big(  x , \Lambda^{k}_{i+1} (t^{k}_{i+1}, x, \lambda) \big) \, \di \Psi^{k}_{i} (x, \lambda)
\\
&
\quad + \int_{Y} \nabla_{\lambda} \varphi  \big ( X^{k}_{i+1}( t, x, \lambda ) , \Lambda^{k}_{i+1} ( t, x, \lambda) \big)\cdot\frac{\big( \Lambda^{k}_{i+1}(t^{k}_{i+1}, x, \lambda) - \lambda \big) }{\tau_{k}}  \, \di \Psi^{k}_{i} (x, \lambda)\,.
\end{split}
\end{equation}

 In order to deduce~\eqref{e:approx_equation2} from~\eqref{e:121}, we need to estimate
\begin{eqnarray*}
&& \displaystyle I_{1} (x, \lambda) \coloneqq \Big| v_{\widetilde{\Psi}^{k}(t)} \big(  x , \Lambda^{k}_{i+1} (t^{k}_{i+1}, x, \lambda) \big) - v_{\Psi^{k}(t)} \big(  X^{k}_{i+1} (t, x, \lambda), \Lambda^{k}_{i+1} (t , x, \lambda) \big) \Big|\,, \\[2mm]
&& \displaystyle I_{2} (x, \lambda) \coloneqq \bigg \| \frac{\big( \Lambda^{k}_{i+1}(t^{k}_{i+1}, x, \lambda) - \lambda \big) }{\tau_{k}} - \T_{\Psi^{k}(t)} ( X^{k}_{i+1} (t, x, \lambda) , \Lambda^{k}_{i+1}(t, x, \lambda) )  \bigg \| _{\mathrm{BL}} 
\end{eqnarray*}
for $(x, \lambda) \in \spt\,  \Psi^{k}_{i} \subseteq \B^{Y}_{R}$, where~$R$ has been determined in Lemma~\ref{l:Gronwall2}.

Let us start with~$I_{1}$. By triangle inequality we have
\begin{equation}\label{e:122}
\begin{split}
I_{1} (x, \lambda)   & \leq   \Big| v_{\widetilde{\Psi}^{k}(t)} \big(  x , \Lambda^{k}_{i+1} (t^{k}_{i+1}, x, \lambda) \big) -  v_{\widetilde{\Psi}^{k}(t)} \big(  X^{k}_{i+1} (t, x, \lambda), \Lambda^{k}_{i+1} (t, x, \lambda) \big) \Big | 
\\
&
 \qquad + \Big |  v_{\widetilde{\Psi}^{k}(t)} \big(  X^{k}_{i+1} (t, x, \lambda), \Lambda^{k}_{i+1} (t, x, \lambda) \big) - v_{\Psi^{k}(t)} \big(  X^{k}_{i+1} (t, x, \lambda), \Lambda^{k}_{i+1} (t , x, \lambda) \big) \Big| 
 \\
 &
\vphantom{\Big|} =: I_{1, 1} (x, \lambda)  + I_{1, 2} (x, \lambda)  \,.
 \end{split}
\end{equation}
Since $\widetilde{\Psi}^{k} (t) \in \Pp( \B^{Y}_{R})$, hypothesis~$(v_{1})$ implies that
\begin{displaymath}
\begin{split}
I_{1,1} (x, \lambda) & \leq \vphantom{\int} L_{v, R} \big( | X^{k}_{i+1}(t, x, \lambda) - x| + \|  \Lambda^{k}_{i+1} (t^{k}_{i+1}, x, \lambda) -  \Lambda^{k}_{i+1} (t, x, \lambda) \|_{\mathrm{BL}} \big)
\\
&
\leq L_{v, R}\bigg( \int_{t^{k}_{i}}^{t} \big | v_{\widetilde{\Psi}^{k}(\tau)} \big (x, \Lambda^{k}_{i+1}(t^{k}_{i+1}, x, \lambda) \big) \big | \, \di \tau +  \int_{t}^{t^{k}_{i+1}} \bigg \| \frac{\Lambda^{k}_{i+1} ( t^{k}_{i+1}, x, \lambda) - \lambda}{\tau_{k}} \bigg  \|_{\mathrm{BL}} \, \di \tau \bigg)  \,.
\end{split}
\end{displaymath}
By~$(v_3)$,  Lemma~\ref{l:Gronwall2}, and Proposition~\ref{p:HS_euler}, we can continue with
\begin{equation}\label{e:123}
\begin{split}
I_{1,1} (x, \lambda) & \leq L_{v, R} \bigg( M_{v} \int_{t^{k}_{i}}^{t} \big( 1 + |x| + \| \Lambda^{k}_{i+1}(t^{k}_{i+1}, x, \lambda) \|_{\mathrm{BL}} + m_{1} ( \widetilde{\Psi}^{k} (\tau)) \big) \, \di \tau 
\\
&
\phantom{ \leq L_{v, R} \bigg(} + \frac{2}{\tau_{k}} \int_{t}^{t^{k}_{i+1}} \HS \big( \Lambda^{k}_{i+1}(t^{k}_{i+1}, x, \lambda), \lambda \big) \, \di \tau \bigg)  \leq C \tau_{k} \,.
\end{split}
\end{equation}

As for $I_{1,2}$, thanks to assumption~$(v_2)$ and to Lemma~\ref{l:Gronwall2} we get
\begin{align*}
I_{1,2}& (x, \lambda)  \leq \vphantom{\int} L_{v, R} W_{1} ( \widetilde{\Psi}^{k}(t) , \Psi^{k}(t) )
\\
&
= L_{v, R} \, \sup_{\eta \in \mathrm{Lip}_{1} ( Y)} \bigg\{ \int_{Y} \eta (x' , \lambda' ) \, \di (\widetilde{\Psi}^{k}(t) - \Psi^{k}(t) ) ( x' , \lambda' )  \bigg\}
\\
&
= L_{v, R} \, \sup_{\eta \in \mathrm{Lip}_{1} ( Y)} \bigg \{ \int_{Y} \eta (x, \Lambda^{k}_{i+1} ( t^{k}_{i+1}, x' , \lambda' ) ) - \eta (X^{k}_{i+1} ( t, x' , \lambda' ), \Lambda^{k}_{i+1} (t, x' , \lambda' ) )  \, \di \Psi^{k}_{i} ( x' , \lambda' ) 
\bigg\}
\\
&
\leq  L_{v, R} \int_{Y} \Big( | x  - X^{k}_{i+1} ( t, x' , \lambda' ) | + \| \Lambda^{k}_{i+1} ( t^{k}_{i+1}, x' , \lambda' ) ) -  \Lambda^{k}_{i+1} (t, x' , \lambda' ) \|_{\mathrm{BL}} \Big)  \, \di \Psi^{k}_{i} ( x' , \lambda' )
\\
&
\leq  L_{v, R} \int_{Y} \bigg( \int_{t^{k}_{i}}^{t} \big | v_{\widetilde{\Psi}^{k}(\tau)} ( x,  \Lambda^{k}_{i+1} ( t^{k}_{i+1}, x' , \lambda' ) )  \big | \, \di \tau  
\\ 
& \qquad\qquad\qquad + \int_{t}^{t^{k}_{i+1}} \bigg \| \frac{\Lambda^{k}_{i+1} ( t^{k}_{i+1}, x' , \lambda' ) - \lambda' }{\tau_{k}} \bigg  \|_{\mathrm{BL}} \, \di \tau \bigg) \, \di \Psi^{k}_{i}( x' , \lambda' )
\\
&
\leq  L_{v, R} \, \tau_{k}  \int_{Y}  \bigg( \big | v_{\widetilde{\Psi}^{k}(\tau)} ( x, \Lambda^{k}_{i+1} ( t^{k}_{i+1}, x' , \lambda' ) )  \big | + \bigg \| \frac{\Lambda^{k}_{i+1} ( t^{k}_{i+1}, x' , \lambda' ) - \lambda}{\tau_{k}} \bigg  \|_{\mathrm{BL}} \bigg) \, \di \Psi^{k}_{i}( x' , \lambda' ) \,.
\end{align*}
Arguing as in~\eqref{e:123} we infer that
\begin{equation}\label{e:124}
I_{1,2} (x, \lambda)  \leq C\, \tau_{k} \qquad \text{for every $(x, \lambda) \in \spt\, \Psi^{k}_{i}$} \,.
\end{equation}
Combining~\eqref{e:122}--\eqref{e:124} we get
\begin{equation}\label{e:125}
I_{1} (x, \lambda) \leq C \, \tau_{k} \qquad \text{for every $(x, \lambda) \in \spt\, \Psi^{k}_{i}$} \,.
\end{equation}

Let us now estimate~$I_{2}$. By triangle inequality we have
\begin{equation}
\begin{split}
I_{2} (x, \lambda)  & \leq    \bigg \| \frac{\big( \Lambda^{k}_{i+1}(t^{k}_{i+1}, x, \lambda) - \lambda \big) }{\tau_{k}} - \T_{\Psi^{k}_{i}} ( x , \Lambda^{k}_{i+1}(t^{k}_{i+1} , x, \lambda) )  \bigg \| _{\mathrm{BL}} 
\\
&
 \quad \vphantom{\int} + \big \| \T_{\Psi^{k}(t) } \big( X^{k}_{i+1} (t, x, \lambda) , \Lambda^{k}_{i+1}(t, x, \lambda) \big) - \T_{\Psi^{k}_{i}} \big( x , \Lambda^{k}_{i+1}(t^{k}_{i+1}, x, \lambda) \big)  \big \| _{\mathrm{BL}} \label{e:126}
 \\
 &
\vphantom{\int} =: I_{2, 1} (x, \lambda)  + I_{2, 2} (x, \lambda)  \,.
 \end{split}
\end{equation}

By Proposition~\ref{p:HS_euler} we have that
\begin{equation}\label{e:127}
I_{2,1} (x, \lambda)  \leq C  \,\tau_{k}  \qquad \text{for every $(x, \lambda) \in \spt\, \Psi^{k}_{i}$} \,.
\end{equation}
 By~$(\T_{2})$,~$(v_{3})$, Lemma~\ref{l:Gronwall2}, and Proposition~\ref{p:HS_euler}, and repeating the arguments of~\eqref{e:124} we get
\begin{align}
\!\! I_{2,2}(x, \lambda)  & \leq L_{\T, R} \bigg( \int_{t^{k}_{i}}^{t} \big| v_{\widetilde{\Psi}^{k}(\tau)} \big( x, \Lambda^{k}_{i+1} ( t^{k}_{i+1}, x, \lambda ) \big) \big| \, \di \tau \nonumber
\\
&
 \quad + \int_{t}^{t^{k}_{i+1}}  \bigg \| \frac{\Lambda^{k}_{i+1} ( t^{k}_{i+1}, x, \lambda) - \lambda}{\tau_{k}} \bigg  \|_{\mathrm{BL}} \di \tau   + W_{1} ( \Psi^{k}(t), \Psi^{k}_{i}) \bigg) \nonumber
\\
&
\leq L_{\T, R} \bigg( \int_{t^{k}_{i}}^{t}  \big| v_{\widetilde{\Psi}^{k}(\tau)} \big( x, \Lambda^{k}_{i+1} ( t^{k}_{i+1}, x, \lambda ) \big) \big| \, \di \tau + \!\int_{t}^{t^{k}_{i+1}} \! \bigg \| \frac{\Lambda^{k}_{i+1} ( t^{k}_{i+1}, x, \lambda) - \lambda}{\tau_{k}} \bigg  \|_{\mathrm{BL}} \di \tau  \label{e:129}
\\
&
 \quad + \int_{Y} \int_{t^{k}_{i}}^{t} \bigg( \big | v_{\widetilde{\Psi}^{k}(\tau)} \big( x', \Lambda^{k}_{i+1} (t^{k}_{i+1}, x', \lambda') \big)\big| + \bigg \| \frac{\Lambda^{k}_{i+1} (t^{k}_{i+1}, x', \lambda') - \lambda' }{\tau_{k}} \bigg\|_{\mathrm{BL}} \bigg)  \di \tau \, \di \Psi^{k}_{i}(x', \lambda') \nonumber
\\
&
\vphantom{\int} \leq 4 L_{\T, R}\, M_{v} (1 + R) \tau_{k} + 4\, \HS \big( \Lambda^{k}_{i+1} ( t^{k}_{i+1}, x, \lambda) , \lambda \big) \leq C\, \tau_{k}\,. \nonumber
\end{align}
 Combining~\eqref{e:126}--\eqref{e:129} we obtain that
\begin{equation}\label{e:130}
I_{2} (x, \lambda) \leq C \, \tau_{k} \qquad \text{for every $(x, \lambda) \in \spt\, \Psi^{k}_{i}$} \,.
\end{equation}
Equality~\eqref{e:approx_equation2} follows from~\eqref{e:125} and~\eqref{e:130} as in the proof of Proposition~\ref{p:discrete_equation}.
 \end{proof}

Finally, we prove the convergence of the sequence~$\Psi^{k}$ to the solution~$\Psi \in C([0,T]; \Pp_{1}(Y))$ of the continuity equation~\eqref{e:cont_eq}.

\begin{theorem}\label{t:limit2}
Let $\widehat{\Psi} \in \Pp_{c}(Y)$. Then, the sequence of curves~$\Psi^{k} \colon [0,T] \to \Pp_{1}(Y)$ converges to the unique solution~$\Psi \in C([0,T]; \Pp_{1}(Y))$ of~\eqref{e:cont_eq} in~$W_{1}$, uniformly with respect to~$t \in [0,T]$.
\end{theorem}

\begin{proof}
Since the operator~$\T_{\Psi}$ defined in~\eqref{e:new_T} satisfies the property~$(\T_{0})$--$(\T_{3})$, we only have to check that the sequence~$\Psi^{k}$ is compact in~$C([0,T]; \Pp_{1}(Y))$. The rest of the proof works as for Theorem~\ref{t:limit} using Proposition~\ref{p:discrete_equation2} instead of Proposition~\ref{p:discrete_equation}.

In view of Lemma~\ref{l:Gronwall2}, it is enough to show that~$\Psi^{k}$ is equi-Lipschitz with respect to~$W_{1}$. Let us fix~$k \in \mathbb{N}$,~$i \in \{0, \ldots, k-1\}$, and~$s \leq t \in [t^{k}_{i}, t^{k}_{i+1}]$. Then,
\begin{displaymath}
\begin{split}
W_{1} & ( \Psi^{k}(t) , \Psi^{k}(s))  = \sup\, \left \{ \int_{Y} \eta (x, \lambda) \, \di ( \Psi^{k}(t) - \Psi^{k}(s))(x, \lambda): \, \eta \in \mathrm{Lip}_{1}(Y) \right \}
\\
&
\leq \int_{Y} \Big( \big| X^{k}_{i+1}(t, x, \lambda) - X^{k}_{i+1} (s, x, \lambda) \big| + \big\| \Lambda^{k}_{i+1}(t, x, \lambda) - \Lambda^{k}_{i+1}(s, x, \lambda) \big\|_{\mathrm{BL}} \Big) \, \di \Psi^{k}_{i}(x, \lambda)
\\
&
\leq \int_{Y} \bigg( \int_{s}^{t} \big| v_{\widetilde{\Psi}^{k}(\tau)} \big(x, \Lambda^{k}_{i+1} (\tau, x, \lambda) \big) \big|\, \di \tau + \int_{s}^{t} \bigg\| \frac{\Lambda^{k}_{i+1}( t^{k}_{i+1}, x, \lambda) - \lambda}{\tau_{k}} \bigg\|_{\mathrm{BL}} \, \di \tau \bigg)\, \di \Psi^{k}_{i}(x, \lambda) \,.
\end{split}
\end{displaymath}
Therefore, by~$(v_2)$, Lemma~\ref{l:Gronwall2}, and Proposition~\ref{p:HS_euler} we get
\begin{displaymath}
W_{1} ( \Psi^{k}(t) , \Psi^{k}(s)) \leq 2M_{v} (1 + R) |t - s| + 2 |t - s| \int_{Y}\frac{\HS \big(( \Lambda^{k}_{i+1}(t^{k}_{i+1}, x, \lambda) , \lambda \big) }{\tau_{k}} \, \di \Psi^{k}_{i}(x, \lambda) \leq C | t - s|
\end{displaymath}
for some positive constant~$C$ independent of~$k$ and~$t$.
\end{proof}


\section{Reversible Markov chains}\label{s:markov}

In this section we show how to adapt the scheme developed in Section~\ref{s:alternative} to a reversible Markov chain on $n$ states. In particular, we will prove the convergence of such scheme for short time.

For fixed $n \in \mathbb{N}$, we consider the set of strategies
\begin{displaymath}
\Lambda_n \coloneqq \bigg\{ \lambda = (\lambda_{1}, \ldots, \lambda_{n}) \in \R^{n} : \, \lambda_{h} >0, \ \sum_{h=1}^{n} \lambda_{h} = 1 \bigg\} \,.
\end{displaymath}
In the notation of Sections~\ref{s:scheme} and~\ref{s:alternative}, the closure~$\overline{\Lambda}_{n}$ can be identified with the set of probability measures~$\Pp(U)$ for~$U \coloneqq \{ e_{h}: \, h = 1, \ldots, n\}$,~$e_{h}$ being the elements of the canonical basis of~$\R^n$. Keeping the notation of the previous sections, we set $Y \coloneqq \R^{d} \times \overline{\Lambda}_{n}$. Furthermore, we define
\begin{align*}
& \Lambda_{n}^{\delta} \coloneqq \{ \lambda \in \Lambda_{n} : \, \lambda_{h} \geq \delta \} \quad \text{for every $\delta>0$}\,, \qquad\R^{n}_{0} \coloneqq \bigg\{ \xi \in \R^{n}: \, \sum_{h = 1}^{n} \xi_{h} = 0 \bigg\}\,,\\
& \B^{Y}_{R, \delta} \coloneqq \B^{Y}_{R} \cap \big( \R^{d} \times \Lambda^{\delta}_{n} \big) \quad \text{for $\delta, \, R > 0$}\,.
\end{align*}

A Markov chain is characterized by a matrix~$\Q \in \mathbb{M}^{n}$, whose element~$\Q_{h \ell} \geq 0$, $h\neq \ell$, indicates the rate of moving from the state~$\ell$ to the state~$h$. In our setting, we consider a more general map~$\mathcal{Q} \colon \R^{d} \times \Pp_{1} (Y) \to \mathbb{M}^{n}$ satisfying the following properties:
\begin{itemize}
\item [$(\Qq_{0})$] for every $(x, \Psi) \in \R^{d} \times \Pp_{1}(Y)$ and every $h, \ell = 1, \ldots, n$, $\Qq_{h \ell}(x, \Psi) \geq 0$ for~$h \neq \ell$, and $\Qq_{hh}(x, \Psi) = - \sum_{\ell \neq h} \Qq_{\ell h} (x, \Psi)$; 

\item[$(\Qq_{1})$] for every $(x, \Psi) \in \R^{d} \times \Pp_{1}(Y)$, $\Qq ( x, \Psi)$ is \emph{reversible}, that is, there exists a unique $\sigma= \sigma (x, \Psi) \in \Lambda_{n}$ such that
\begin{displaymath}
\Qq_{h \ell} ( x, \Psi ) \sigma_{\ell} = \Qq_{\ell h} (x, \Psi) \sigma_{h} \qquad \text{for every $h, \ell = 1, \ldots, n$}\,;
\end{displaymath}

\item[$(\Qq_{2})$] $\Qq$ is locally Lipschitz, that is, for every $R>0$ there exists $L_{\Qq, R}>0$ such that for every $x_{1}, x_{2} \in \B_{R}$ and every $\Psi_{1}, \Psi_{2} \in \Pp (\B^{Y}_{R})$
\begin{displaymath}
| \Qq( x_{1}, \Psi_{1}) - \Qq (x_{2}, \Psi_{2}) | \leq L_{\Qq, R} \big( | x_{1} - x_{2} | + W_{1} (\Psi_{1}, \Psi_{2}) \big)\,;
\end{displaymath}

\item[$(\Qq_{3})$] there exists $M_{\Qq} > 0$ such that for every $x \in \R^{d}$ and every $\Psi \in \Pp_{1}(Y)$
\begin{displaymath}
| \Qq ( x, \Psi) | \leq M_{\Qq} \big( 1 + |x| + m_{1} ( \Psi ) \big) \,.
\end{displaymath}
\end{itemize}

\begin{remark}
We remark that~$(\Qq_{1})$ is always satisfied, for instance, when~$\Qq (x, \Psi)$ is a tridiagonal matrix for every $x \in \R^{d}$ and~$\Psi \in \Pp_{1}(Y)$, see, e.g.,~\cite[Section~5.1]{Mielke}.
\end{remark}

\begin{remark}\label{r:QT}
We notice that if for every $y = (x, \lambda)  \in Y$ and every $\Psi \in \Pp_{1}(Y)$ we set $\T_{\Psi} (y)\coloneqq \Qq ( x, \Psi) \lambda$, then the operator~$\T \colon Y \times \Pp_{1}(Y) \to \Lambda_{n}$ satisfies properties~$(\T_{0})$--$(\T_{3})$ of Section~\ref{s:hp}.
\end{remark}

Following~\cite{Maas,Mielke}, for every $y = (x,\lambda) \in \R^{d} \times \Lambda_{n}$ and every $\Psi \in \Pp_{1}(Y)$ we consider the \emph{entropy} $E$ and the \emph{Onsager matrix} $K$
\begin{align}
& E ( x, \lambda, \Psi) \coloneqq \sum_{h  = 1}^{n} \lambda_{h} \ln \bigg( \frac{\lambda_{h}}{\sigma_{h} (x, \Psi)} \bigg)\,, \label{e:E} \\[1mm]
& K(x, \lambda, \Psi) \coloneqq \sum_{\ell = 2}^{n} \sum_{h = 1}^{\ell - 1} \Qq_{h \ell}(x, \Psi) \sigma_{\ell}(x, \Psi) \, \Phi \bigg( \frac{\lambda_{h}}{\sigma_{h}( x, \Psi) } , \frac{\lambda_{\ell}}{\sigma_{\ell}(x, \Psi)} \bigg) ( e_{h} - e_{\ell}) \otimes (e_{h} - e_{\ell}) \,, \label{e:K}
\end{align}
where $\Phi \colon [0,+\infty) \times [0,+\infty) \to [0,+\infty)$ is defined as
\begin{displaymath}
\Phi ( a, b ) \coloneqq \frac{ a - b}{\ln a - \ln b }\quad \text{for $a \neq b$,} \qquad \Phi (a, a) = a\,,
\end{displaymath}
so that $\Phi$ is analytic. Clearly, $E(x, \cdot, \Psi)$ and $K(x, \cdot, \Psi)$ can be extended to~$\overline{\Lambda}_{n}$ by continuity. Moreover, we notice that for every $(x, \lambda) \in \R^{d} \times \Lambda_{n}$ and every $\Psi \in \Pp_{1}(Y)$, the matrix~$K(x, \lambda, \Psi)$ is symmetric and positive definite when acting on~$\R^{n}_{0}$. We denote by~$G(x, \lambda, \Psi)$ its inverse on~$\R^{n}_{0}$. The matrix~$G$ is a Riemannian tensor on~$\R^{n}_{0}$. For every $x \in \R^{d}$ and $\Psi \in \Pp_{1}(Y)$ we define the Riemannian metric $\sfd_{(x, \Psi)} \colon \Lambda_{n} \times \Lambda_{n} \to [0,+\infty)$ as
 \begin{align}\label{e:d}
\sfd_{(x, \Psi)} ( \lambda_{1}, \lambda_{2} ) \coloneqq \inf \, \bigg\{ \int_{0}^{1} \left\langle G(x, \rho(s), \Psi) \rho'(s) , \rho'(s) \right\rangle^{\frac12}\di s :&\,  \rho \in C^{1} ( [0 , 1 ] ; \Lambda_{n} ) , 
\\
&
\, \rho (0) = \lambda_{1} , \, \rho (1) = \lambda_{2} \bigg\} \,, \nonumber
\end{align}
for every $\lambda_{1}, \lambda_{2} \in \Lambda_{n}$. The metric $\sfd_{(x, \Psi)}$ can be extended to~$\overline{\Lambda}_{n} \times \overline{\Lambda}_{n}$ in a continuous way.

In the next two lemmas we collect some properties of~$E$,~$K$,~$G$, and~$\sfd_{(x,\Psi)}$.

\begin{lemma}\label{l:estimate}
Let~$\delta, R > 0$. Then, the following facts hold:
\begin{itemize}
\item[$(i)$] there exists a positive constant~$\eta = \eta (R)$ such that $\sigma_{h} (x, \Psi) \geq \eta$ for every $x \in \B_{R}$, every $\Psi \in \Pp(\B^{Y}_{R})$, and every $h = 1, \ldots, n$;

\item[$(ii)$] there exist two positive constants~$c_{1} = c_{1}( R)$ and~$ c_{2} =  c_{2}(R)$ such that for every $x \in \B_{R}$, every $\lambda \in \Lambda_{n}$, every $\Psi \in \Pp(\B^{Y}_{R})$, and every $\mu \in \R^{n}_{0}$,
\begin{align}
& c_{1} | \mu | ^{2} \leq \left\langle G(x, \lambda, \Psi) \mu, \mu \right\rangle \,, \label{e:G1}\\[1mm]
& \left\langle K(x, \lambda, \Psi) \mu, \mu \right\rangle \leq c_{2}  | \mu | ^{2}\, \label{e:K1};
\end{align}

\item[$(iii)$] there exist two positive constants~$c_{3} = c_{3}( \delta, R)$ and $c_{4} = c_{4}(\delta, R)$ such that for every $(x, \lambda) \in \B^{Y}_{R, \delta}$, every $\Psi \in \Pp(\B^{Y}_{R})$, and every $\mu \in \R^{n}_{0}$,
\begin{align}
&  \left\langle G(x, \lambda, \Psi) \mu, \mu \right\rangle \leq c_{3} | \mu | ^{2}\,, \label{e:G2}\\[1mm]
&  c_{4}  | \mu | ^{2} \leq \left\langle K(x, \lambda, \Psi) \mu, \mu \right\rangle \, \label{e:K2};
\end{align}

\item[$(iv)$] $G(x, \cdot, \Psi)$ is Lipschitz continuous in~$\Lambda^{\delta}_{n}$, uniformly with respect to~$x \in \B_{R}$ and~$\Psi \in \Pp(\B^{Y}_{R})$, that is, there exists $L_{G, \delta, R}>0$ such that for every $\lambda_{1}, \lambda_{2} \in \Lambda^{\delta}_{n}$
\begin{equation}\label{e:G3}
| G(x, \lambda_{1}, \Psi) - G(x, \lambda_{2}, \Psi) | \leq L_{G, \delta, R} | \lambda_{1} - \lambda_{2}| \,;
\end{equation}

\item[$(v)$] $E (x, \cdot, \Psi)$ is Lipschitz continuous in~$\Lambda^{\delta}_{n}$, uniformly with respect to~$x \in \B_{R}$ and~$\Psi \in \Pp(\B^{Y}_{R})$, namely, there exists $L_{E, \delta, R}>0$ such that for every $\lambda_{1}, \lambda_{2} \in \Lambda^{\delta}_{n}$
\begin{equation}\label{e:E1}
| E(x, \lambda_{1}, \Psi) - E(x, \lambda_{2}, \Psi) | \leq L_{E, \delta, R} | \lambda_{1} - \lambda_{2}| \,;
\end{equation}

\item[$(vi)$] for every $\alpha \in (0,1)$ the energy~$E(x, \cdot, \Psi)$ is $\alpha$-H\"older continuous in~$\overline{\Lambda}_{n}$, uniformly with respect to~$x \in \B_{R}$ and~$\Psi \in \Pp (\B^{Y}_{R})$, that is, for every $\alpha \in (0,1)$ there exists $C_{E, \alpha, R}>0$ such that for every $\lambda_{1}, \lambda_{2} \in \overline{\Lambda}_{n}$
\begin{equation}\label{e:E2}
| E(x, \lambda_{1}, \Psi) - E(x, \lambda_{2}, \Psi) | \leq C_{E, \alpha, R} | \lambda_{1} - \lambda_{2}|^{\alpha} \,.
\end{equation}
\end{itemize}
\end{lemma}

\begin{remark}
The constants~$c_{1}(R)$ and~$c_{4}(\delta, R)$ can be assumed to be decreasing with respect to~$R$, while~$c_{2}( R)$,~$c_{3}(\delta, R)$,~$L_{G, \delta, R}$,~$L_{E, \delta, R}$, and~$C_{E, \alpha, R}$ can be assumed to be increasing with respect to~$R$. 
\end{remark}

\begin{proof}[Proof of Lemma~\ref{l:estimate}]
In view of~$(\Qq_{1})$ and~$(\Qq_{2})$, we have that the function~$(x, \Psi) \mapsto \sigma (x, \Psi)$ is continuous from~$\R^{d} \times \Pp_{1}(Y) \to \Lambda_{n}$. Hence, there exists~$\eta= \eta( R ) > 0$ such that for every $x \in \B_{R}$, every $\Psi \in \Pp_{1}(\B^{Y}_{R})$, and every $h \in \{ 1, \ldots, n\}$, $ \sigma_{h} (x, \Psi) \geq \eta >0$, so that $(i)$ holds.

From~$(i)$,~\eqref{e:K}, the regularity of~$\Phi$, and~$(\Qq_{3})$, we further deduce that~\eqref{e:K1} holds for a suitable constant~$c_{2} = c_{2}( R)$. 

For every $(x , \lambda) \in \R^{d} \times \Lambda_{n}$ and every $\Psi \in \Pp_{1}(Y)$ we have that $K( x, \lambda, \Psi)$ is symmetric, positive semi-definite on~$\R^{n}$, and positive definite on~$\R^{n}_{0}$. Since $K$ is continuous with respect to~$(x, \lambda, \Psi)$, we deduce that there exists a positive constant~$c_{4} = c_{4} (\delta, R) \leq c_{2}$ such that  inequality~\eqref{e:K2} holds for every $(x , \lambda) \in \B^{Y}_{R, \delta}$ and every $\Psi \in \Pp (\B^{Y}_{R})$. Since $G$ is the inverse of~$K$ on~$\R^{n}_{0}$,~\eqref{e:K1} and~\eqref{e:K2} imply~\eqref{e:G1} and~\eqref{e:G2} with~$c_{1}(R) \coloneqq  c_{2}(R) ^{-1}  $ and $c_{3}(\delta, R) \coloneqq c_{4} (\delta, R)^{-1}$. This concludes the proof of~$(ii)$ and~$(iii)$.

The Lipschitz continuity $(iv)$ of~$G(x, \cdot, \Psi)$ in~$\Lambda^{\delta}_{n}$ follows from the regularity of~$K(x, \cdot, \Psi)$, from~$(i)$--$(iii)$, and from the identity
\begin{displaymath}
G ( x , \lambda_{1}, \Psi) - G(x, \lambda_{2}, \Psi) = G ( x , \lambda_{1}, \Psi) \big ( K ( x , \lambda_{2}, \Psi) - K ( x , \lambda_{1}, \Psi)\big) G ( x , \lambda_{2}, \Psi)  \qquad \text{on~$\R^{n}_{0}$}\,.
\end{displaymath}

As for~$(v)$, we notice that for $x \in \B_{R}$,~$\Psi \in \Pp(\B^{Y}_{R})$, and~$\lambda \in \Lambda^{\delta}_{n}$, the ratio $\lambda_{h}/\sigma_{h} (x, \Psi)$ is bounded from below and from above by~$ \delta$ and by~$1/\eta$, respectively. Since the function~$a \mapsto a \ln a$ is locally Lipschitz continuous in~$(0,+\infty)$, we have that there exists $L = L ( \delta, R)>0$ such that~\eqref{e:E1} holds.

Finally, we note that the function~$a \mapsto a \ln a$ belongs to~$W^{1, p}( [ 0, A])$ for every $p \in [1, +\infty)$ and every $A<+\infty$. In view of~$(i)$, for every $x \in \B_{R}$, every $\Psi \in \Pp(\B^{Y}_{R})$, and every $\lambda \in \overline{\Lambda}_{n}$, the ratio $\lambda_{h}/\sigma_{h} (x, \Psi)$ is bounded above by~$1/\eta$. Hence, by Sobolev embedding in dimension one we infer that for every $\alpha \in (0, 1)$ there exists $C = C ( \alpha, R) > 0$ such that~\eqref{e:E2} holds. 
\end{proof}

Before stating the main properties of the distance~$\sfd_{(x,\Psi)}$, we define, for every $x \in \R^{d}$, every $\Psi \in \Pp_{1}(Y)$, and every $\lambda, \lambda_{1}, \lambda_{2} \in \Lambda_{n}$, the norm
\begin{displaymath}
\| \lambda_{1} - \lambda_{2}\|_{G(x, \lambda, \Psi)} \coloneqq \left \langle G(x, \lambda, \Psi) (\lambda_{1} - \lambda_{2}) , \lambda_{1} - \lambda_{2} \right \rangle ^{\frac12} \,,
\end{displaymath}
which is well-defined in view of~\eqref{e:G1} and~\eqref{e:G2}.

\begin{lemma}\label{l:distance}
Let $\delta,  R>0$ and let~$c_{1}, c_{3} >0$ be the constants determined in~\eqref{e:G1} and~\eqref{e:G2}. Then, the following facts hold:
\begin{itemize}
\item[$(i)$] there exists a positive constant~$m_{1} = m_{1} (R)$ such that for every $x \in \B_{R}$ and every $\Psi \in \Pp(\B^{Y}_{R})$
\begin{equation}\label{e:d1}
m_{1} | \lambda_{1} - \lambda_{2} | \leq \sfd_{(x, \Psi)} ( \lambda_{1}, \lambda_{2}) \qquad \text{for every $\lambda_{1}, \lambda_{2} \in \Lambda_{n}$}\,;
\end{equation}


\item[$(ii)$] there exist two positive constants $m_{2} = m_{2} (\delta, R)$ and $m_{3} = m_{3}(\delta, R)$ such that for every $x \in \B_{R}$, every $\Psi \in \Pp(B^{Y}_{R})$, and every $\lambda_{1}, \lambda_{2} \in \Lambda^{\delta}_{n}$
\begin{align}
& \sfd_{(x, \Psi)} ( \lambda_{1} , \lambda_{2}) \leq m_{2} | \lambda_{1} - \lambda_{2} | \,, \label{e:d5} \\
& \sfd_{(x, \Psi)} (\lambda_{1}, \lambda_{2}) \leq  \| \lambda_{1} - \lambda_{2} \|_{G(x, \lambda_{1}, \Psi)} + m_{3} | \lambda_{1} - \lambda_{2} |^{\frac32}; \label{e:d3}
\end{align}

\item[$(iii)$] there exists a positive constant $m_{4}= m_{4}(\delta, R)$ such that for every $x \in \B_{R}$, every $\Psi \in \Pp(B^{Y}_{R})$, and every $\lambda_{1}, \lambda_{2} \in \Lambda^{\delta}_{n}$ satisfying
\begin{equation}\label{e:134}
\sqrt{\frac{c_{3}}{c_{1}}}\, | \lambda_{1} - \lambda_{2} | < \min\, \big\{ \dist (\lambda_{1}, \partial\Lambda^{\delta}_{n}), \dist (\lambda_{2}, \partial\Lambda^{\delta}_{n}) \big\}
\end{equation}
we have
\begin{equation}
  \| \lambda_{1} - \lambda_{2} \|_{G(x, \lambda_{1}, \Psi)} \leq  \sfd_{(x, \Psi)} (\lambda_{1}, \lambda_{2}) + m_{4} | \lambda_{1} - \lambda_{2} |^{\frac32}. \label{e:d4}
\end{equation}
\end{itemize}
\end{lemma}

\begin{remark}
The constant~$m_{1}(R)$ can be assumed to be decreasing with respect to~$R$, while~$m_{2}(\delta, R)$, $m_{3}(\delta, R)$, and~$m_{4}(\delta, R)$, can be assumed to be increasing with respect to~$R$. 
\end{remark}

\begin{proof}[Proof of Lemma~\ref{l:distance}]
Point~$(i)$ is a consequence of~\eqref{e:G1}. We now prove~$(ii)$. Given~$x \in \B_{R}$, $\Psi \in \Pp( \B^{Y}_{R})$, and~$\lambda_{1}, \lambda_{2} \in \Lambda^{\delta}_{n}$, we have that the curve
\begin{displaymath}
\rho(s) \coloneqq (1- s) \lambda_{1} + s\lambda_{2} \qquad s \in [0,1]
\end{displaymath}
is a competitor for the infimum in the definition of~$\sfd_{(x, \Psi)} (\lambda_{1}, \lambda_{2})$ in~\eqref{e:d}. Moreover, by convexity, $\rho(s) \in \Lambda^{\delta}_{n}$ for every $s \in [0,1]$. Therefore, applying $(iii)$ of Lemma~\ref{l:estimate} we get
\begin{displaymath}
\sfd_{(x, \Psi)} (\lambda_{1}, \lambda_{2})  \leq \int_{0}^{1} \left\langle G(x, \rho(s) , \Psi) (\lambda_{2} - \lambda_{1}), \lambda_{2} - \lambda_{1}\right\rangle ^{\frac12} \di s  \leq \sqrt{c_{3}} | \lambda_{1} - \lambda_{2} | \,,
\end{displaymath}
which is~\eqref{e:d5} with $m_{2} = \sqrt{c_{3}}$. 

Combining, instead,~$(iii)$ and~$(iv)$ of Lemma~\ref{l:estimate} we can further estimate
\begin{equation}\label{e:131}
\begin{split}
\sfd_{(x, \Psi)} (\lambda_{1}, \lambda_{2}) & \leq \int_{0}^{1} \left\langle G(x, \rho(s) , \Psi) (\lambda_{2} - \lambda_{1}), \lambda_{2} - \lambda_{1}\right\rangle ^{\frac12} \di s  
\\
&
\vphantom{\int} \leq \left\langle G(x, \lambda_{1} , \Psi) (\lambda_{2} - \lambda_{1}), \lambda_{2} - \lambda_{1}\right\rangle ^{\frac12} 
\\
&
\qquad + \int_{0}^{1} \left |\left\langle \big( G(x, \rho(s) , \Psi) - G(x, \lambda_{1}, \Psi) \big) (\lambda_{2} - \lambda_{1}), \lambda_{2} - \lambda_{1}  \right\rangle \right| ^{\frac12} \di s 
\\
&
\leq  \| \lambda_{1} - \lambda_{2} \|_{ G (x, \lambda_{1}, \Psi) } + \int_{0}^{1} \left( L_{ G, \delta, R } | \lambda_{1} - \rho(s) | \right )^{\frac{1}{2}} | \lambda_{1} - \lambda_{2} | \, \di s 
\\
&
\vphantom{\int} \leq \| \lambda_{1} - \lambda_{2} \|_{G(x, \lambda_{1}, \Psi)} + \sqrt{L_{G, \delta, R}}\, |\lambda_{1} - \lambda_{2} | ^{\frac{3}{2}},
\end{split}
\end{equation}
from which we conclude~\eqref{e:d3} with $m_{3} = \sqrt{L_{G, \delta, R}}$.

Finally, let $x$, $\Psi$, $\lambda_{1}$, and~$\lambda_{2}$ be as in point~$(iv)$. For every $\varepsilon>0$ let $\rho_{\varepsilon} \in C^{1}([0,1]; \Lambda_{n})$ with $\rho_{\varepsilon} (0) = \lambda_{1}$ and $\rho_{\varepsilon} (1) = \lambda_{2}$ be such that
\begin{equation}\label{e:132}
\int_{0}^{1} \left\langle G(x, \rho_{\varepsilon}(s) ,\Psi)  \rho_{\varepsilon}'(s) , \rho_{\varepsilon}'(s) \right\rangle^{\frac{1}{2}} \di s \leq \sfd_{(x, \Psi)} (\lambda_{1} , \lambda_{2}) + \varepsilon\,.
\end{equation}
In view of~\eqref{e:d5} and of~\eqref{e:G1}, we deduce from~\eqref{e:132} that
\begin{equation}\label{e:133}
\sqrt{c_{1}} \int_{0}^{1} | \rho_{\varepsilon}'(s) | \, \di s \leq m_{2} | \lambda_{1} - \lambda_{2} | + \varepsilon = \sqrt{c_{3}} | \lambda_{1} - \lambda_{2} | + \varepsilon \,.
\end{equation}
Hence,~\eqref{e:134} and~\eqref{e:133} imply that
\begin{displaymath}
 \int_{0}^{1} | \rho_{\varepsilon}'(s) | \, \di s < \min\, \big\{ \dist (\lambda_{1}, \partial\Lambda^{\delta}_{n}), \dist (\lambda_{2}, \partial\Lambda^{\delta}_{n}) \big\} + \frac{\varepsilon}{\sqrt{c_{1}}}\,.
\end{displaymath}
Therefore, for~$\varepsilon$ small enough we may assume that $\rho_{\varepsilon}(s) \in \Lambda^{\delta}_{n}$ for every $s \in [0,1]$. For such $\varepsilon$ we estimate
\begin{displaymath}
\begin{split}
\| \lambda_{1} - \lambda_{2} \|_{G(x, \lambda_{1}, \Psi)} & \leq \int_{0}^{1} \left\langle G(x, \lambda_{1}, \Psi) \rho_{\varepsilon}'(s) , \rho'_{\varepsilon} (s) \right\rangle ^{\frac{1}{2}} \di s 
\\
&
\leq \int_{0}^{1} \left\langle G(x, \rho_{\varepsilon}(s), \Psi) \rho_{\varepsilon}'(s) , \rho_{\varepsilon}' (s) \right\rangle^{\frac{1}{2}} \di s 
\\
&
\qquad + \int_{0}^{1}\left| \left\langle \big( G(x, \lambda_{1}, \Psi) - G( x, \rho_{\varepsilon}(s), \Psi) \big) \rho_{\varepsilon}'(s) , \rho_{\varepsilon}' (s) \right\rangle \right|^{\frac{1}{2}} \di s 
\\
&
\leq \sfd_{(x, \Psi)} (\lambda_{1}, \lambda_{2}) + \int_{0}^{1} \left| \left\langle \big( G(x, \lambda_{1}, \Psi) - G( x, \rho_{\varepsilon}(s), \Psi) \big) \rho_{\varepsilon}'(s) , \rho_{\varepsilon}' (s) \right\rangle \right|^{\frac{1}{2}} \di s  + \varepsilon\,.
\end{split}
\end{displaymath}
Since $\rho_{\varepsilon}(s) \in \Lambda^{\delta}_{n}$ for every $s \in [0,1]$, by~$(iv)$ of Lemma~\ref{l:estimate} and by~\eqref{e:133} we have that
\begin{equation}\label{e:135}
\begin{split}
\| \lambda_{1} - \lambda_{2} \|_{G(x, \lambda_{1}, \Psi)} & \leq \sfd_{(x, \Psi)} (\lambda_{1}, \lambda_{2}) + \sqrt{L_{G, \delta, R}} \int_{0}^{1} | \lambda_{1} - \rho_{\varepsilon}(s) | ^{\frac12} | \rho_{\varepsilon} ' (s) | \, \di s + \varepsilon
\\
&
\leq \sfd_{(x, \Psi)} (\lambda_{1}, \lambda_{2}) + \sqrt{L_{G, \delta, R}} \bigg(\int_{0}^{1} | \rho_{\varepsilon} ' (s) | \, \di s \bigg)^{\frac32} + \varepsilon
\\
&
\leq \sfd_{(x, \Psi)} (\lambda_{1}, \lambda_{2}) + \sqrt{L_{G, \delta, R}}\, \bigg( \frac{c_{3}}{c_{1}} \bigg)^{\frac34} |\lambda_{1} - \lambda_{2}| ^{\frac32}+ \varepsilon \bigg( 1 + \sqrt{\frac{ L_{ G, \delta, R}}{c_{1}^{3/2}}} \bigg).
\end{split}
\end{equation}
Thus, we conclude~\eqref{e:d4} by passing to the limit in~\eqref{e:135} as~$\varepsilon \to 0$. In particular,~$m_{4} = \sqrt{L_{G, \delta, R}} \big( \frac{c_{3}}{c_{1}} \big)^{\frac34}$.
\end{proof}

We now rewrite the alternate scheme presented in Section~\ref{s:alternative} in the language of Markov chains, and show its short-time convergence to a solution of the continuity equation~\eqref{e:cont_eq}, where for~$\Psi \in \Pp_{1}(Y)$ the field $b_{\Psi} \colon Y \to Y $ is now defined as
\begin{displaymath}
b_{\Psi} ( x, \lambda) \coloneqq \left( \begin{array}{cc}
v_{\Psi} (x, \lambda) \\
\Qq(x, \Psi) \lambda
\end{array}\right)
\end{displaymath}
for a velocity field $v_{\Psi} \colon Y \to \R^{d}$ satisfying properties~$(v_{1})$--$(v_{3})$ of Section~\ref{s:hp}.

Let us fix a time step~$\tau_{k}>0$, $k \in \mathbb{N}$, such that $\tau_{k} \to 0 $ as $k \to \infty$, and let $t^{k}_{i} \coloneqq i \tau_{k}$ for $i \in \mathbb{N}$. For $i=0$ we set $\Psi^{k}_{0} \coloneqq \widehat{\Psi} \in \Pp_{1}(Y)$. For $i>0$, assume we are given $\Psi^{k}_{i} \in \Pp_{1}(Y)$. Then, similarly to~\eqref{e:HS_step}, the label of an agent sitting in position~$\hat{x} \in \R^{d}$ with label~$\hat{\lambda} \in \overline{\Lambda}_{n}$ is updated by solving the minimizing movement
\begin{equation}\label{e:Markov_step}
\min \, \bigg\{ E ( \hat{x} , \lambda , \Psi^{k}_{i} ) + \frac{1}{2\tau_{k}} \, \sfd^{2}_{(\hat{x}, \Psi^{k}_{i})} ( \lambda, \hat{\lambda}) : \, \lambda \in \overline{\Lambda}_{n} \bigg\} \,.
\end{equation}
Since~$\overline{\Lambda}_{n}$ is compact,~\eqref{e:Markov_step} admits at least one solution~$\lambda_{(\hat{x}, \hat{\lambda}), i+1}$\footnote{The arguments in~\cite[Section~2.3]{Mielke} can also be used to show that~\eqref{e:Markov_step} admits indeed a unique solution for~$\tau_{k}$ sufficiently small. However, uniqueness is not needed in our framework.}.
Therefore, we can define~$\lambda^{k}_{(\hat{x}, \hat{\lambda}), i+1}$, $\Lambda^{k}_{i+1}$, and~$\tilde{\Psi}^{k}_{i+1}$ exactly as in~\eqref{100},~\eqref{e:lambda}, and~\eqref{e:tilde_Psi}, respectively. The step~\eqref{e:ODE2} in the space variable remains the same, and~$x^{k}_{(\hat{x}, \hat{\lambda}), i+1}$, $X^{k}_{i+1}$, $\Psi^{k}_{i+1}$ are as in~\eqref{101},~\eqref{e:X}, and~\eqref{e:Psi}. Furthermore, we refer to~\eqref{e:Psi},~\eqref{e:Psi_tilde}, and~\eqref{e:underline_Psi} for the definition of the interpolation curves~$\Psi^{k}$,~$\widetilde{\Psi}^{k}$, and~$\underline{\Psi}^{k}$.

Repeating step by step the proofs of Lemmas~\ref{l:Gronwall_bound} and~\ref{l:Gronwall2}, we obtain the following uniform estimate on~$\Psi^{k}(t)$,~$\widetilde{\Psi}^{k}(t)$, and~$\underline{\Psi}^{k}(t)$.

\begin{lemma}\label{l:Gronwall3}
Let $\widehat{\Psi} \in \Pp_{c}(Y)$. Then there exists an increasing continuous function $R\colon[0,+\infty) \to [0,+\infty)$ such that for every $T \in [0,+\infty)$, every $k \in \mathbb{N}$, and every $t\in [0,T]$, $\Psi^{k}(t),  \underline{\Psi}^{k}(t),  \widetilde{\Psi}^{k}(t) \in \Pp(\B^{Y}_{R(T)})$.
\end{lemma}

\begin{proof}
The statement follows by the arguments of Lemmas~\ref{l:Gronwall_bound} and~\ref{l:Gronwall2}. In particular, we gave there an explicit formula for~$R(T)$ as a function of~$T \in [0,+\infty)$, which turns out to be continuous and increasing. 
\end{proof}

Also in the current setting, we need to write an approximate Euler-Lagrange equation associated with~\eqref{e:Markov_step}. This is done in  Proposition~\ref{p:Markov_euler} below, for proving which we need the following lemma.
\begin{lemma}\label{lemmafacile}
Let $f\colon \R^{N}\to\R\cup\{+\infty\}$ be a convex function, let $A\in \mathbb{M}^N$ be a symmetric and positive definite matrix, and let $\lVert\cdot\rVert_A\colon\R^{N}\to[0,+\infty)$ be the norm associated with $A$, namely $\lVert \xi\rVert_A^2\coloneqq \langle A\xi,\xi\rangle$, for all $\xi\in\R^{N}$. For a fixed $\zeta\in\R^{N}$ and $c>0$, assume that $\xi_0$ solves
\begin{equation}\label{900}
\min\big\{f(\xi)+c\lVert \xi-\zeta \rVert_A^2\big\}.
\end{equation}
Then $\xi_0$ also solves
\begin{equation}\label{901}
\min\big\{f(\xi)+c\lVert \xi-\zeta \rVert_A^2-c\lVert \xi-\xi_0 \rVert_A^2 \big\}.
\end{equation}
\end{lemma}
\begin{proof}
It is enough to observe that the problem \eqref{901} can be equivalently rewritten as
$$\min\big\{f(\xi)+2c \langle\xi, A(\xi_0-\zeta)\rangle \big\}$$
hence it is a convex minimization problem. Since $\xi_0$ solves \eqref{900}, we have~$-2cA(\xi_0-\zeta)\in\partial f(\xi_0)$, which is exactly the Euler condition for the problem above.

\end{proof}

\begin{proposition}\label{p:Markov_euler}
Let $\delta, R > 0$ and let~$m_{1}(R)$, $C_{E, \alpha, R}$, and $L_{E, \delta, R}$ be the constants determined in Lemmas~\ref{l:estimate} and~\ref{l:distance}. Assume that $\Psi \in \Pp (\B^{Y}_{R})$ and $( x, \lambda) \in \B^{Y}_{R, \delta}$, and let~$\tilde{\lambda}$  be a solution to
\begin{equation}\label{e:140}
\min \, \bigg\{E  (x , \rho, \Psi ) + \frac{1}{2 \tau_{k} } \sfd^{2}_{(x, \Psi)} (\rho, \lambda) : \, \rho \in \overline{\Lambda}_{n} \bigg\} \,.
\end{equation}
Then, the following facts hold:
\begin{itemize}
\item[$(i)$] for every $\alpha \in (0, 1)$
\begin{equation}\label{e:140.1}
| \tilde \lambda - \lambda| \leq \bigg( \frac{2 \, C_{E, \alpha, R}}{m_{1}^{2}}\bigg)^{1/( 2 - \alpha)}\, \tau_{k}^{1/( 2 - \alpha)} \,;
\end{equation}

\item[$(ii)$] if $\tilde{\lambda} \in \Lambda^{\delta}_{n}$, then
\begin{equation}\label{e:140.2}
|\tilde{\lambda} - \lambda | \leq \frac{2 \, L_{E, \delta, R}}{m_{1}^{2}} \, \tau_{k}\,;
\end{equation}

\item[$(iii)$] if~$\lambda, \tilde{\lambda} \in \Lambda^{\delta}_{n}$ and $\mu$ is the unique solution to
\begin{equation}\label{e:145}
\min\, \bigg\{ E(x, \rho, \Psi) + \frac{1}{2\tau_{k}} \| \rho - \lambda\| _{G(x, \tilde{\lambda}, \Psi)}^{2} : \, \rho \in \overline{\Lambda}_{n} \bigg\} \,,
\end{equation}
then, for every $\alpha \in (0,1)$ we have
\begin{equation}\label{e:140.3}
| \mu - \lambda| \leq \bigg( \frac{2 \, C_{E, \alpha, R}}{m_{1}^{2}}\bigg)^{1/( 2 - \alpha)}\, \tau_{k}^{1/( 2 - \alpha)} \,.
\end{equation}
If, in addition, $\mu \in \Lambda^{\delta}_{n}$, then
\begin{equation}\label{e:146}
| \mu - \lambda | \leq \frac{2\, L_{E, \delta, R}}{m_{1}^{2}} \, \tau_{k}\,.
\end{equation}
Finally, if $\lambda, \tilde{\lambda} \in \Lambda^{\delta}_{n}$ satisfy~\eqref{e:134}, there exists a positive constant $C = C (\delta, R)$ such that
\begin{align}
& \bigg | \frac{\tilde{\lambda} - \lambda}{ \tau_{k}} - \Qq (x, \Psi)\tilde{\lambda}  \bigg | \leq C \tau_{k}^{1/4} \,. \label{e:142}
\end{align}
\end{itemize}
\end{proposition}

\begin{proof}
By the minimality of~$\tilde{\lambda}$, by~$(vi)$ of Lemma~\ref{l:estimate}, and by~$(i)$ of Lemma~\ref{l:distance} we have that for every $\alpha \in (0, 1)$
\begin{equation}\label{e:145.1}
\frac{m_{1}^{2}}{2\tau_{k}} | \tilde{\lambda} - \lambda|^{2} \leq \big| E(x, \lambda, \Psi) - E(x, \tilde \lambda, \Psi) \big | \leq C_{E, \alpha, R} | \tilde{\lambda} - \lambda| ^{\alpha},
\end{equation}
where~$m_{1} = m_{1}(R)$ and~$C_{E, \alpha, R}$ are defined in Lemmas~\ref{l:estimate} and~\ref{l:distance}, respectively. From~\eqref{e:145.1} we deduce~\eqref{e:140.1}. In a similar way we deduce~\eqref{e:140.3}, recalling that $m_{1} = \sqrt{c_{1}}$, where~$c_{1}$ has been determined in~\eqref{e:G1}.

If we further assume that~$\tilde\lambda \in \Lambda^{\delta}_{n}$, by minimality of~$\tilde\lambda$, by $(v)$ of Lemma~\ref{l:estimate}, and by~$(i)$ of Lemma~\ref{l:distance}, we have that
\begin{displaymath}
\frac{m^{2}_{1}}{2\tau_{k}} | \lambda - \tilde{\lambda} |^{2} \leq  \frac{1}{2\tau_{k}} \, \sfd^{2}_{(x, \Psi)} (\lambda, \tilde{\lambda} ) \leq | E(x, \lambda, \Psi) - E(x, \tilde{\lambda}, \Psi) | \leq L_{E, \delta, R} | \lambda - \tilde{\lambda} |\,.
\end{displaymath} 
Hence, we deduce~\eqref{e:140.2}. Moreover, if $\mu \in \Lambda^{\delta}_{n}$, in the very same way we get~\eqref{e:146}.

In order to prove~\eqref{e:142}, we first estimate the Euclidean norm $| \mu - \tilde{\lambda}|$. 
Denote by $\chi_{\overline{\Lambda}_{n}}$ the characteristic function of the convex set $\overline{\Lambda}_{n}$ in the sense of convex analysis.
Since $E(x, \cdot, \Psi)$ is convex in~$\overline{\Lambda}_{n}$, we can apply Lemma~\ref{lemmafacile} with $f(\cdot)=E(x, \cdot, \Psi)+\chi_{\overline{\Lambda}_{n}}(\cdot)$, $\xi_0=\mu$, $c=\frac1{2\tau_k}$, $\zeta=\lambda$, and $A=G(x, \tilde{\lambda}, \Psi)$ obtaining
\begin{displaymath}
\begin{split}
E(x, \mu, \Psi) + \frac{1}{2\tau_{k}} \| \mu - \lambda \|_{G(x, \tilde{\lambda}, \Psi)}^{2} + \frac{1}{2\tau_{k}} \| \mu - \tilde{\lambda} \|_{G(x, \tilde{\lambda}, \Psi)}^{2} \leq E(x, \tilde{\lambda}, \Psi) + \frac{1}{2\tau_{k}} \| \tilde{\lambda} - \lambda\|_{G(x, \tilde{\lambda}, \Psi) }^{2} \,.
\end{split}
\end{displaymath}
Re-ordering the terms in the previous inequality and adding and subtracting on the right-hand side the terms $\frac{1}{2\tau_{k}}  \sfd^{2}_{(x, \Psi)} (\tilde{\lambda}, \lambda)$ and~$\frac{1}{2\tau_{k}}  \sfd^{2}_{(x, \Psi)} ( \mu , \lambda)$ we obtain
\begin{equation}\label{e:147}
\begin{split}
\frac{1}{2\tau_{k}} \| \mu - \tilde{\lambda} \|_{G(x, \tilde{\lambda}, \Psi)}^{2}  \leq & \ E(x, \tilde{\lambda}, \Psi) + \frac{1}{2\tau_{k}}  \sfd^{2}_{(x, \Psi)} (\tilde{\lambda}, \lambda) - E(x, \mu, \Psi) -  \frac{1}{2\tau_{k}}  \sfd^{2}_{(x, \Psi)} ( \mu , \lambda)  
\\
&
- \frac{1}{2\tau_{k}} \| \mu - \lambda \|_{G(x, \tilde{\lambda}, \Psi)}^{2}+ \frac{1}{2\tau_{k}} \| \tilde{\lambda} - \lambda\|_{G(x, \tilde{\lambda}, \Psi) }^{2} 
\\
&
+ \frac{1}{2\tau_{k}}  \sfd^{2}_{(x, \Psi)} ( \mu , \lambda) - \frac{1}{2\tau_{k}}  \sfd^{2}_{(x, \Psi)} (\tilde{\lambda}, \lambda)\,.
\end{split}
\end{equation}
By the minimality of~$\tilde{\lambda}$, inequality~\eqref{e:147} simplifies to
\begin{equation}\label{e:148}
\| \mu - \tilde{\lambda} \|_{G(x, \tilde{\lambda}, \Psi)}^{2}  \leq   \| \tilde{\lambda} - \lambda\|_{G(x, \tilde{\lambda}, \Psi) }^{2}  - \| \mu - \lambda \|_{G(x, \tilde{\lambda}, \Psi)}^{2}
 +   \sfd^{2}_{(x, \Psi)} ( \mu , \lambda) -  \sfd^{2}_{(x, \Psi)} (\tilde{\lambda}, \lambda)\,.
\end{equation}
Since $x \in \B_{R}$, $\Psi \in \Pp(\B^{Y}_{R})$, $\lambda, \tilde{\lambda}, \mu \in \Lambda^{\delta}_{n}$, and~$\lambda, \tilde\lambda$ satisfy~\eqref{e:134}, we deduce from~\eqref{e:148}, from~$(ii)$ of Lemma~\ref{l:estimate}, and from $(ii)$--$(iii)$ of Lemma~\ref{l:distance} that
\begin{equation}\label{e:149}
\begin{split}
c_{1}^{2} | \mu - \tilde\lambda | ^{2} \leq & \ \left( d_{(x, \Psi)} (\tilde \lambda, \lambda) + m_{4} | \tilde \lambda - \lambda |^{\frac32} \right) ^{2} + \left( \| \mu - \lambda \|_{G(x, \tilde\lambda, \Psi)} + m_{3} | \mu - \lambda |^{\frac32} \right)^{2}
\\
&
 \vphantom{\Big(} - \| \mu - \lambda \|_{G(x, \tilde{\lambda}, \Psi)}^{2} - \sfd^{2}_{(x, \Psi)} (\tilde{\lambda}, \lambda)\,.
\end{split}
\end{equation}
Developing the squares and using~$(iii)$ of Lemma~\ref{l:estimate} and~$(ii)$ of Lemma~\ref{l:distance}, we continue in~\eqref{e:149} with
\begin{equation}\label{e:150}
c_{1}^{2} | \mu - \tilde \lambda | ^{2} \leq   m_{4}^{2}  | \tilde{\lambda} - \lambda | ^{3} +  m_{3}^{2} | \mu - \lambda |^{3} + 2 \,m_{2} \, m_{4} | \tilde{\lambda} - \lambda |^{\frac52} + 2\, \sqrt{c_{3}} \, m_{3} | \mu - \lambda |^{\frac52} \,.
\end{equation}
Combining~\eqref{e:150} with~\eqref{e:140.2} and~\eqref{e:146} we deduce
\begin{equation}\label{e:153}
| \mu - \tilde \lambda | \leq \widetilde{C} \tau_{k}^{5/4} \,.
\end{equation}
for some positive constant~$\widetilde{C} = \widetilde{C} (\delta, R)$ independent of~$k$.

We are now in a position to conclude~\eqref{e:142}. The minimality of~$\mu$, indeed, implies that for every $\xi \in \R^{n}_{0}$
\begin{displaymath}
\left\langle D_{\lambda} E(x, \mu, \Psi) , \xi \right\rangle + \frac{1}{\tau_{k}} \big\langle G(x, \tilde{\lambda}, \Psi) (\mu - \lambda) , \xi \big\rangle = 0 \,.
\end{displaymath}
By a simple algebraic manipulation, we rewrite the previous equality as
\begin{equation}\label{e:151}
\begin{split}
\left\langle D_{\lambda} E(x, \mu, \Psi) , \xi \right\rangle & + \frac{1}{\tau_{k}} \big\langle G(x, \mu , \Psi) (\tilde{\lambda} - \lambda) , \xi \big\rangle 
\\
&
\hspace{-5mm} = \frac{1}{\tau_{k}}  \big\langle G(x, \tilde\lambda , \Psi) ( \tilde{\lambda} - \mu ) , \xi \big\rangle 
+ \frac{1}{\tau_{k}}  \big\langle \big(G(x, \mu , \Psi) - G(x, \tilde{\lambda}, \Psi) \big) ( \tilde{\lambda} - \lambda ) , \xi \big\rangle.
\end{split}
\end{equation}
Taking~$\xi = K^{\top}(x, \mu, \Psi) \omega$ for $\omega \in \R^{n}$ in~\eqref{e:151}, we get that
\begin{displaymath}
\begin{split}
K(x, \mu, \Psi)\,D_{\lambda} E(x, \mu, \Psi) + \frac{1}{\tau_k} ( \tilde{\lambda} - \lambda)= & \frac{1}{\tau_{k}}  K(x, \mu, \Psi)\, G(x, \tilde\lambda , \Psi) ( \tilde{\lambda} - \mu )
\\
& + \frac{1}{\tau_{k}}   K(x, \mu, \Psi)\,\big(G(x, \mu , \Psi) - G(x, \tilde{\lambda}, \Psi) \big) ( \tilde{\lambda} - \lambda ) .
\end{split}
\end{displaymath}
Since $K(x, \mu, \Psi) D_{\lambda} E(x, \mu, \Psi) = - \Qq (x, \Psi) \mu$ (see~\cite[Theorem~3.1]{Mielke}), we actually have
\begin{equation}\label{e:152}
\begin{split}
  - \Qq(x, \Psi) \tilde\lambda + \frac{1}{\tau_k} \big( \tilde{\lambda} - \lambda\big) = &  \frac{1}{\tau_{k}}   K(x, \mu, \Psi)\, G(x, \tilde\lambda , \Psi)  \big( \tilde{\lambda} - \mu \big)  
\\
&
 + \frac{1}{\tau_{k}}  K(x, \mu, \Psi)\,\big(G(x, \mu , \Psi) - G(x, \tilde{\lambda}, \Psi) \big) ( \tilde{\lambda} - \lambda ) 
 \\
 &
   \vphantom{\frac12}+  \Qq(x, \Psi) ( \tilde\lambda - \mu)  \,.
\end{split}
\end{equation}
Combining~$(ii)$--$(iv)$ of Lemma~\ref{l:estimate} with the inequalities~\eqref{e:140.2},~\eqref{e:146},~\eqref{e:153}, and~\eqref{e:152}, and with the assumptions~$x \in \B_{R}$, $\Psi \in \Pp(\B^{Y}_{R})$, and $ \tilde{\lambda}, \mu \in \Lambda^{\delta}_{n}$, we get~\eqref{e:142}, and therefore the proof is concluded.
\end{proof}

\begin{lemma}\label{l:time}
Let $r > 0$, $\eta > \delta >0$, and $\widehat{\Psi} \in \Pp( \B^{Y}_{r, \eta})$. Then, there exists $T_{f} >0$ such that for every $k \in \mathbb{N}$ large enough and every $t < T_{f}$ the following hold:
\begin{itemize}
\item[$(i)$] $\Psi^{k}(t), \underline{\Psi}^{k}(t), \widetilde{\Psi}^{k}(t) \in \Pp (\B^{Y}_{R(T_{f}), \delta})$, where~$R \colon [0,+\infty) \to [0, +\infty)$ is the function determined in Lemma~\ref{l:Gronwall3};

\item[$(ii)$] if $t \in[t^{k}_{i}, t^{k}_{i + 1})$ for some $i \in \mathbb{N}$, for every $(x, \lambda) \in \spt\,  \underline{\Psi}^{k} (t)$
\begin{displaymath}
\sqrt{\frac{c_{3}( \delta, R(T_{f}))}{c_{1}(R(T_{f}))}} \, | \Lambda^{k}_{i+1} (t, x, \lambda) - \lambda | < \min \, \left \{ \dist (\lambda, \partial\Lambda^{\delta}_{n}) , \dist (\Lambda^{k}_{i+1} (t, x, \lambda), \partial\Lambda^{\delta}_{n}) \right \}\,.
\end{displaymath}
\end{itemize}
\end{lemma}

\begin{proof}
Since $\widehat{\Psi} \in \Pp (\B_{r, \eta}^{Y})$, we deduce from Lemma~\ref{l:Gronwall3} that for every $T>0$, every $k \in \mathbb{N}$, and every $i$ such that $i \tau_{k} \leq T$ we have~$\Psi^{k}_{i}, \widetilde{\Psi}^{k}_{i} \in \Pp (\B^{Y}_{R(T)})$. Hence, in order to conclude the lemma we have to study the behavior of the labels~$\lambda \in \overline{\Lambda}_{n}$ along the alternating scheme.

Along the proof of the lemma we denote by~$\lambda^{k}(t, x_{0}, \lambda_{0})$ and~$x^{k}(t, x_{0}, \lambda_{0})$, for $(x_{0}, \lambda_{0}) \in \spt\, \widehat{\Psi}$,  the curves obtained by iteratively solving~\eqref{e:Markov_step} and the difference equation~\eqref{e:ODE2} in each interval~$[t^{k}_{i}, t^{k}_{i+1}]$ starting from~$(x_{0}, \lambda_{0})$ at time $t_0=0$ and using, at each node~$t^{k}_{i}$, $\hat \lambda = \lambda^{k}(t^{k}_{i}, x_{0}, \lambda_{0})$ and~$\hat x = x^{k}(t^{k}_{i}, x_{0}, \lambda_{0})$ as new initial conditions. As in~\eqref{e:110}, we define the piecewise constant interpolations~$\overline{x}^{k} (t, x_{0}, \lambda_{0}), \underline{x}^{k} (t, x_{0}, \lambda_{0})$ and~$\overline{\lambda}^{k} (t, x_{0}, \lambda_{0}), \underline{\lambda}^{k}(t, x_{0}, \lambda_{0})$.

The assumption~$\widehat{\Psi} \in \Pp(\B^{Y}_{r, \eta})$ means that for every $(x_{0}, \lambda_{0}) \in \spt \widehat{\Psi}$ we have $\lambda_{0} \in \Lambda^{\eta}_{n}$. Since the measures~$\Psi^{k}(t)$ and~$\widetilde{\Psi}^{k}(t)$ are supported on pairs of the form $(x^{k} (t, x_{0}, \lambda_{0}), \lambda^{k} (t, x_{0}, \lambda_{0}))$ and $(\underline{x}^{k}(t, x_{0}, \lambda_{0}), \overline{\lambda}^{k}(t, x_{0}, \lambda_{0}))$, respectively, we are led to estimate the number of steps needed by~\eqref{e:Markov_step} to exit~$\Lambda^{\delta}_{n}$, knowing that the initial label~$\lambda_{0} \in \Lambda^{\eta}_{n}$. 

Let us fix~$\overline{\alpha} \in (0,1)$. For~$k \in \mathbb{N}$ such that $\tau_{k} \leq 1$, we claim that the properties~$(i)$ and~$(ii)$ hold with $R = R( t^{k}_{i})$ for every $t \in [0, t^{k}_{i}]$ until the following conditions are fulfilled:
\begin{align}
& \sum_{j = i}^{i-1} \frac{2 \, L(j-1, k)}{m_{1}^{2} (j-1, k)} \, \tau_{k} + \left( \frac{2 \, C(i-1, k) }{m_{1}^{2}(i-1, k)}  \right)^{1/(2 - \overline \alpha)}  \sqrt{\tau_{k}} < \eta - \delta \,,\label{e:160} \\[1mm]
& \sum_{j = i}^{i-1} \frac{2 \, L(j-1, k)}{m_{1}^{2} (j-1, k)} \, \tau_{k} + \frac{2 \, C(i-1, k)}{m_{1}^{2}(i-1, k)}  \left(\frac{c_{3} (i-1, k)}{c_{1}(i-1, k) }\right)^{1/2} \tau_{k} < \eta - \delta\,, \label{e:162}
\end{align}
where we have set $L(j, k) \coloneqq L_{E, \delta, R(j\tau_{k})}$, $C(j, k) \coloneqq C_{E, \overline \alpha, R(j\tau_{k})}$, $m_{1}(j, k) \coloneqq m_{1}(R(j\tau_{k}))$, $c_{1}(j, k) \coloneqq c_{1} (R(j\tau_{k}))$, and $c_{3}(j, k) \coloneqq c_{3}(\delta, R(j\tau_{k}))$. 

Given the claim for granted, for every $k \in \mathbb{N}$ let us denote with~$i_{k} \in \mathbb{N}$ the first index for which at least one of the two inequalities~\eqref{e:160} or \eqref{e:162} is violated. For simplicity, let us assume that it is always~\eqref{e:160} to be violated in~$i_{k}$. Hence,
\begin{displaymath}
 \sum_{j =1}^{i_{k} - 1} \frac{2 \, L(j-1, k)}{m_{1}^{2} (j-1, k)} \, \tau_{k} \geq \eta - \delta - \left( \frac{2 \, C(i_{k} - 1, k) }{m_{1}^{2}(i_{k}-1, k)}  \right) ^{1/(2 - \overline{\alpha})}  \sqrt{\tau_{k}} \,.
\end{displaymath}
Since $L_{E, \delta, R}$ is increasing with respect to~$R$,~$m_{1}(R)$ is decreasing with respect to~$R$, and~$R(t)$ determined in Lemma~\ref{l:Gronwall3} is increasing with respect to~$t$, we also have that
\begin{displaymath}
\frac{2 \, L(i_{k}-1, k)}{m_{1}^{2} (i_{k}-1, k)} \, (i_{k} - 1) \tau_{k} \geq \eta - \delta - \left( \frac{2 \, C(i_{k} - 1, k) }{m_{1}^{2}(i_{k}-1, k)}  \right) ^{1/(2 - \overline{\alpha})}   \sqrt{\tau_{k}}\,,
\end{displaymath}
from which we deduce that $i_{k} \tau_{k}$ is bounded away from $0$. Therefore, there exists $T_{f}>0$ such that $T_{f} < (i_{k} - 1) \tau_{k}$ for every $k$ large enough. A similar estimate can be obtained if~\eqref{e:162} is violated, and we conclude that there exists~$T_{f}>0$ such that $(i)$ and~$(ii)$ hold.

Let us prove the claim. For fixed~$i \in \mathbb{N}$, assume that \eqref{e:160}--\eqref{e:162} hold and that $\lambda^{k}(t^{k}_{j}, x_{0}, \lambda_{0}) \in \Lambda^{\delta}_{n}$ for $0 \leq j<i$. Then, by~$(ii)$ of Proposition~\ref{p:Markov_euler} we have that for every $1 \leq j < i$
\begin{equation}\label{e:163}
| \lambda^{k} (t^{k}_{j} , x_{0}, \lambda_{0}) - \lambda^{k} (t^{k}_{j-1}, x_{0}, \lambda_{0}) | \leq \frac{2 \, L(j-1, k)}{m_{1}^{2} (j-1, k)} \, \tau_{k} \,.
\end{equation} 
By~$(i)$ of Proposition~\ref{p:Markov_euler} we have, since $\tau_{k} \leq 1$,
\begin{equation}\label{e:164}
| \lambda^{k} (t^{k}_{i} , x_{0}, \lambda_{0}) - \lambda^{k} (t^{k}_{i-1}, x_{0}, \lambda_{0}) | \leq \left( \frac{2 \, C(i-1, k) }{m_{1}^{2}(i-1, k)}  \right) ^{1/(2 - \overline{\alpha})}  \sqrt{\tau_{k}} \,.
\end{equation}
Hence, by~\eqref{e:160},~\eqref{e:163},~\eqref{e:164}, and by triangle inequality, we deduce that
\begin{displaymath}
\begin{split}
| \lambda^{k} (t^{k}_{i} , x_{0}, \lambda_{0}) - \lambda_{0} | & \leq \sum_{j=1}^{i} | \lambda^{k} (t^{k}_{j} , x_{0}, \lambda_{0}) - \lambda^{k} (t^{k}_{j-1}, x_{0}, \lambda_{0}) | 
\\
&
\leq \sum_{j = i}^{i-1} \frac{2 \, L(j-1, k)}{m_{1}^{2} (j-1, k)} \, \tau_{k} + \left( \frac{2 \, C(i-1, k) }{m_{1}^{2}(i-1, k)}  \right) ^{1/(2 - \overline{\alpha})}  \sqrt{\tau_{k}} < \eta - \delta \,,
\end{split}
\end{displaymath}
which implies that $\lambda^{k} (t^{k}_{i}, x_{0}, \lambda_{0}) \in \Lambda^{\delta}_{n}$ as $\lambda_{0} \in \Lambda^{\eta}_{n}$. Since~\eqref{e:160} is independent of the particular choice of the initial condition~$(x_{0}, \lambda_{0}) \in \spt\, \widehat{\Psi} \subseteq \B^{Y}_{r, \eta}$, we infer that $\Psi^{k}(t), \underline{\Psi}^{k} (t), \widetilde{\Psi}^{k} (t) \in \Pp(\B^{Y}_{R(t^{k}_{i}), \delta})$ for every $t \in [0, t^{k}_{i}]$.

Let us now denote by~$\mu^{k}_{i} \in \overline{\Lambda}_{n}$ the solution to the minimum problem
\begin{displaymath}
\min_{\rho \in \overline{\Lambda}_{n}}\, \bigg\{ E \left (x^{k} (t^{k}_{i-1}, x_{0}, \lambda_{0}), \rho, \Psi^{k}_{i-1} \right) + \frac{1}{2 \tau_{k}} \| \rho - \lambda^{k} (t^{k}_{i-1}, x_{0}, \lambda_{0}) \|^{2}_{G(x^{k} (t^{k}_{i-1}, x_{0}, \lambda_{0}) , \lambda^{k}(t^{k}_{i}, x_{0}, \lambda_{0}) , \Psi^{k}_{i-1})} \bigg\} \,.
\end{displaymath}
Then, by~$(iii)$ of Proposition~\ref{p:Markov_euler} we get that
\begin{displaymath}
| \mu^{k}_{i} - \lambda^{k} ( t^{k}_{i-1} , x_{0}, \lambda_{0}) | \leq \left( \frac{2 \, C(i-1, k)}{m^{2}_{1}(i-1, k)}  \right) ^{1/(2 - \overline{\alpha})}   \sqrt{\tau_{k}} \,.
\end{displaymath}
Therefore, by triangle inequality and by~\eqref{e:160} we obtain
\begin{displaymath}
\begin{split}
| \mu^{k}_{i} - \lambda_{0} | & \leq \sum_{j=1}^{i-1} | \lambda^{k} (t^{k}_{j}, x_{0}, \lambda_{0} ) - \lambda^{k}( t^{k}_{j-1}, x_{0}, \lambda_{0}) | + | \mu^{k}_{i} - \lambda^{k} (t^{k}_{i-1}, x_{0}, \lambda_{0}) |
\\
&
\leq \sum_{j = i}^{i-1}\frac{2 \, L(j-1, k)}{m_{1}^{2} (j-1, k)} \, \tau_{k} + \left( \frac{2 \, C(i-1, k)}{m_{1}^{2} (i-1, k)}  \right) ^{1/(2 - \overline{\alpha})}  \sqrt{\tau_{k}} < \eta - \delta \,,
\end{split}
\end{displaymath}
which yields $\mu^{k}_{i} \in \Lambda^{\delta}_{n}$.

Since $\lambda^{k} (t^{k}_{j}, x_{0}, \lambda_{0} ) \in \Lambda^{\delta}_{n}$ for~$0\leq j \leq i$, by~$(ii)$ of Proposition~\ref{p:Markov_euler}, by~\eqref{e:162}, and by~\eqref{e:163} we have that
\begin{displaymath}
\begin{split}
\sum_{j=1}^{i-1}  | \lambda^{k} (t^{k}_{j}, x_{0}, \lambda_{0})  & - \lambda^{k} (t^{k}_{j-1}, x_{0}, \lambda_{0}) | +  \left(\frac{c_{3} (i-1, k)}{c_{1}(i-1, k) }\right)^{1/2} | \lambda^{k} (t^{k}_{i}, x_{0}, \lambda_{0} ) - \lambda^{k}( t^{k}_{i-1}, x_{0}, \lambda_{0}) | \\
&
\leq  \sum_{j = i}^{i-1} \frac{2 \, L(j-1, k)}{m_{1}^{2} (j-1, k)} \, \tau_{k} + \frac{2 \, C(i-1, k)}{m_{1}^{2}(i-1, k)}  \left(\frac{c_{3} (i-1, k)}{c_{1}(i-1, k) }\right)^{1/2} \tau_{k} < \eta - \delta \,,
\end{split}
\end{displaymath}
which in turn implies
\begin{equation} \label{e:165}
\begin{split}
\left(\frac{c_{3} (i-1, k)}{c_{1}(i-1, k) }\right)^{1/2} & | \lambda^{k} (t^{k}_{i}, x_{0}, \lambda_{0} ) - \lambda^{k}( t^{k}_{i-1}, x_{0}, \lambda_{0}) | 
\\
&
< \min\, \Big \{ \dist \big( \lambda^{k} (t^{k}_{i}, x_{0}, \lambda_{0} ), \partial \Lambda^{\delta}_{n} \big)  , \dist \big ( \lambda^{k} (t^{k}_{i-1}, x_{0}, \lambda_{0} ), \partial \Lambda^{\delta}_{n} \big) \Big \} \,.
\end{split}
\end{equation}
Since all the estimates above are independent of the particular choice of~$(x_{0}, \lambda_{0}) \in \spt\, \widehat{\Psi}$ and since, for $t \in [t^{k}_{i-1}, t^{k}_{i})$, the measure~$\underline{\Psi}^{k}(t)$ has support
\begin{displaymath}
\spt \, \underline{\Psi}^{k}(t) \subseteq \Big\{ \big(x^{k}(t^{k}_{i-1}, x_{0}, \lambda_{0}) , \lambda^{k}( t^{k}_{i-1}, x_{0}, \lambda_{0}) \big) : \, (x_{0}, \lambda_{0}) \in \spt \, \widehat{\Psi} \Big\} \subseteq \B^{Y}_{R (t^{k}_{i-1} ) , \delta}\,,
\end{displaymath}
 we deduce that~$(ii)$ holds.
\end{proof}

We are now in a position to prove the short-time convergence of the alternate Lagrangian scheme for reversible Markov chains. We start by showing the equivalent of Propositions~\ref{p:discrete_equation} and~\ref{p:discrete_equation2}.

\begin{proposition}\label{p:Markov_discrete}
Let $r>0$, $\eta > \delta > 0$, $\widehat{\Psi} \in \Pp (\B^{Y}_{r, \eta})$, and let~$T_{f}>0$ be as in Lemma~\ref{l:time}. Then, there exists~$C > 0$ such that for every $\varphi \in C_{b}^{1}( \R^{d} \times \overline{\Lambda}_{n})$, every~$k \in \mathbb{N}$ large enough, every $i \in \mathbb{N}$ such that $(i + 1) \tau_{k} <T_{f}$, and every $t \in (t^{k}_{i}, t^{k}_{i+1})$,
\begin{equation}\label{e:approx_equation3}
\frac{\di}{\di t} \int_{Y} \varphi (x, \lambda) \, \di \Psi^{k}(t)(x, \lambda) = \int_{Y} \nabla \varphi (x, \lambda) \cdot b_{\Psi^{k}(t)} (x, \lambda) \, \di \Psi^{k}(t) (x, \lambda) + \vartheta_{k}(\varphi) \,,
\end{equation}
where $|\vartheta_{k}(\varphi)|  \leq C \| \varphi \|_{C^{1}_{b}} \tau^{1/4}_{k}$.
\end{proposition}

\begin{proof}
Along the proof we denote by~$C$ a generic positive constant independent of~$i$,~$k$,~$t$, and~$\varphi$, that may vary from line to line.

We follow step by step the proof of Propositions~\ref{p:discrete_equation} and~\ref{p:discrete_equation2}. Let $i$ and~$k$ be as in the statement of the proposition, and let us set $R \coloneqq R(T_{f})$. For every test function $\varphi \in C_{b}^{1}(\R^{d} \times \Lambda_{n})$ and every $t \in (t^{k}_{i}, t^{k}_{i+1})$, by definition of~$\Psi^{k}(t)$ we have that
\begin{equation}\label{e:166}
\begin{split}
\frac{\di}{\di t} & \int_{Y} \varphi (x, \lambda) \, \di \Psi^{k}(t)(x, \lambda) = \frac{\di}{\di t} \int_{Y} \varphi \big ( X^{k}_{i+1}( t, x, \lambda ) , \Lambda^{k}_{i+1} ( t, x, \lambda) \big) \, \di \Psi^{k}_{i} (x, \lambda) 
\\
&
= \int_{Y} \nabla_{x} \varphi \big ( X^{k}_{i+1}( t, x, \lambda ) , \Lambda^{k}_{i+1} ( t, x, \lambda) \big) \cdot v_{\widetilde{\Psi}^{k} (t)} \big(  x , \Lambda^{k}_{i+1} (t^{k}_{i+1}, x, \lambda) \big) \, \di \Psi^{k}_{i} (x, \lambda)
\\
&
\quad + \int_{Y} \nabla_{\lambda} \varphi  \big ( X^{k}_{i+1}( t, x, \lambda ) , \Lambda^{k}_{i+1} ( t, x, \lambda) \big) \cdot \frac{\big( \Lambda^{k}_{i+1}(t^{k}_{i+1}, x, \lambda) - \lambda \big) }{\tau_{k}}  \, \di \Psi^{k}_{i} (x, \lambda)\,.
\end{split}
\end{equation}

 In order to deduce~\eqref{e:approx_equation3} from~\eqref{e:166} we estimate
\begin{eqnarray*}
&& \displaystyle I_{1} (x, \lambda) \coloneqq \Big| v_{\widetilde{\Psi}^{k}(t)} \big(  x , \Lambda^{k}_{i+1} (t^{k}_{i+1}, x, \lambda) \big) - v_{\Psi^{k}(t)} \big(  X^{k}_{i+1} (t, x, \lambda), \Lambda^{k}_{i+1} (t , x, \lambda) \big) \Big|\,, \\[2mm]
&& \displaystyle I_{2} (x, \lambda) \coloneqq \bigg| \frac{\big( \Lambda^{k}_{i+1}(t^{k}_{i+1}, x, \lambda) - \lambda \big) }{\tau_{k}} - \Qq \big( X^{k}_{i+1} (t, x, \lambda) , \Psi^{k}(t) \big) \Lambda^{k}_{i+1}(t, x, \lambda)  \bigg | 
\end{eqnarray*}
for $(x, \lambda) \in \spt  \Psi^{k}_{i} \subseteq \B^{Y}_{R, \delta}$, the last inclusion being a consequence of Lemma~\ref{l:time}.

Arguing as in the proof of~\eqref{e:122}--\eqref{e:124} and using~\eqref{e:140.2} we get that
\begin{align}
I_{1}  & \leq  L_{v, R} \int_{t^{k}_{i}}^{t^{k}_{i+1}} \bigg ( M_{v}  \Big( 1 + |x| + | \Lambda^{k}_{i+1}(t^{k}_{i+1}, x, \lambda) | + m_{1} ( \widetilde{\Psi}^{k} (\tau)) \Big) + \bigg | \frac{\Lambda^{k}_{i+1} ( t^{k}_{i+1}, x, \lambda) - \lambda}{\tau_{k}} \bigg  | \bigg) \di \tau \nonumber
\\
&
\quad + L_{v, R} \, \tau_{k}  \int_{Y}  \bigg( \big | v_{\widetilde{\Psi}^{k}(\tau)} ( x', \Lambda^{k}_{i+1} ( t^{k}_{i+1}, x' , \lambda' ) )  \big | + \bigg | \frac{\Lambda^{k}_{i+1} ( t^{k}_{i+1}, x', \lambda' ) - \lambda' }{\tau_{k}} \bigg  | \bigg)  \di \Psi^{k}_{i}( x' , \lambda' ) \label{e:167}
\\
&
\vphantom{\int_{Y}} \leq 4\,L_{v, R} \, M_{v} ( 1 + R) \tau_{k} + \frac{4 \, L_{E, \delta, R}}{m_{1}^{2}} \tau_{k} = C \, \tau_{k}\,. \nonumber
\end{align}

Let us now estimate~$I_{2}$. By triangle inequality we have
\begin{equation}
\begin{split}
I_{2}  & \leq    \bigg | \frac{\big( \Lambda^{k}_{i+1}(t^{k}_{i+1}, x, \lambda) - \lambda \big) }{\tau_{k}} - \Qq \big( x , \Psi^{k}_{i} \big)  \Lambda^{k}_{i+1}(t^{k}_{i+1} , x, \lambda)  \bigg | 
\\
&
 \qquad \vphantom{\int} + \big | \Qq \big( X^{k}_{i+1} (t, x, \lambda) , \Psi^{k}(t) \big) \Lambda^{k}_{i+1}(t^{k}_{i+1}, x, \lambda) - \Qq \big( x , \Psi^{k}_{i} \big)  \Lambda^{k}_{i+1}(t, x, \lambda)   \big |  \label{e:168}
 =: I_{2, 1} + I_{2, 2} \,.
 \end{split}
\end{equation}

By~$(iii)$ of Proposition~\ref{p:Markov_euler} and by Lemma~\ref{l:time} we have that
\begin{equation}\label{e:169}
I_{2,1}  \leq C  \,\tau_{k}^{1/4} \,.
\end{equation}
 By~$(\Qq_{2})$,~$(v_{3})$, Lemmas~\ref{l:Gronwall3} and~\ref{l:time}, and~$(ii)$ of Proposition~\ref{p:Markov_euler}, we get
\begin{equation}\label{e:170}
\begin{split}
I_{2,2} & \leq L_{\Qq, R} \bigg( \int_{t^{k}_{i}}^{t} \big| v_{\widetilde{\Psi}^{k}(\tau)} \big( x, \Lambda^{k}_{i+1} ( t^{k}_{i+1}, x, \lambda ) \big) \big| \, \di \tau + \int_{t}^{t^{k}_{i+1}}  \bigg | \frac{\Lambda^{k}_{i+1} ( t^{k}_{i+1}, x, \lambda) - \lambda}{\tau_{k}} \bigg  | \di \tau \bigg)
\\
&
\vphantom{\int} \leq 2 L_{\Qq, R}\, M_{v} (1 + R)\,  \tau_{k} + \frac{2 \, L_{E, \delta, R}}{m_{1}^{2}} \, \tau_{k}  = C\, \tau_{k}\,.
\end{split}
\end{equation}
 Combining~\eqref{e:169} and \eqref{e:170} we obtain that
\begin{equation}\label{e:171}
I_{2} \leq C \, \tau_{k}\,.
\end{equation}
Finally, equality~\eqref{e:approx_equation3} follows from~\eqref{e:167} and~\eqref{e:171} as in the proof of Propositions~\ref{p:discrete_equation} and~\ref{p:discrete_equation2}.
\end{proof}

We finally conclude with the main result of this section.

\begin{theorem}\label{t:Markov}
Let $r>0$, $\eta > \delta > 0$, and $\widehat{\Psi} \in \Pp(\B^{Y}_{r, \eta})$. Then, there exists $T_{f}>0$ such that the sequence of curves~$\Psi^{k} \colon [0,T_{f}] \to \Pp_{1}(Y)$ converges to the unique solution~$\Psi \in C([0,T_f]; \Pp_{1}(Y))$ of~\eqref{e:cont_eq} in~$W_{1}$, uniformly with respect to~$t \in [0,T_f]$ 
\end{theorem}

\begin{proof}
Let $T_{f}>0$ be as in Lemma~\ref{l:time}, so that the curves $\Psi^{k}, \widetilde{\Psi}^{k}$, and~$\underline{\Psi}^{k}$ are well defined in the interval~$[0,T_{f}]$ and~\eqref{e:approx_equation3} holds. Since the operator~$\T_{\Psi}( x, \lambda) \coloneqq \Qq (x, \Psi) \lambda $ satisfies the property~$(\T_{0})$--$(\T_{3})$, we only have to check that the sequence~$\Psi^{k}$ is compact in~$C([0,T_{f}]; \Pp_{1}(Y))$. The rest of the proof works as for Theorem~\ref{t:limit}, with the obvious modifications due to the fact that the rest~$\theta_{k}$ in Proposition~\ref{p:Markov_discrete} is now controlled by~$\tau_{k}^{1/4}$ and not by~$\tau_{k}$.

In view of Lemma~\ref{l:time}, it is enough to show that~$\Psi^{k}$ is equi-Lipschitz with respect to~$W_{1}$. Let us fix~$k \in \mathbb{N}$,~$i \in \mathbb{N}$ such that $i \tau_{k} \leq T_{f}$, and~$s \leq t \in [t^{k}_{i}, t^{k}_{i+1}]$, and let $R \coloneqq R(T_{f})$. Then,
\begin{displaymath}
\begin{split}
W_{1} & ( \Psi^{k}(t) , \Psi^{k}(s))  = \sup_{\eta \in \mathrm{Lip}_{1}(Y)} \bigg\{ \int_{Y} \eta (x, \lambda) \, \di ( \Psi^{k}(t) - \Psi^{k}(s))(x, \lambda) \bigg\}
\\
&
\leq \int_{Y} \Big( \big| X^{k}_{i+1}(t, x, \lambda) - X^{k}_{i+1} (s, x, \lambda) \big| + \big | \Lambda^{k}_{i+1}(t, x, \lambda) - \Lambda^{k}_{i+1}(s, x, \lambda) \big | \Big) \, \di \Psi^{k}_{i} (x, \lambda)
\\
&
\leq \int_{Y} \bigg( \int_{s}^{t} \big| v_{\widetilde{\Psi}^{k}(\tau)} \big(x, \Lambda^{k}_{i+1} (\tau, x, \lambda) \big) \big|\, \di \tau + \int_{s}^{t} \bigg | \frac{\Lambda^{k}_{i+1}( t^{k}_{i+1}, x, \lambda) - \lambda}{\tau_{k}} \bigg | \, \di \tau \bigg)\, \di \Psi^{k}_{i}(x, \lambda) \,.
\end{split}
\end{displaymath}
Therefore, by~$(v_2)$, Lemma~\ref{l:time}, and Proposition~\ref{p:Markov_euler} we get
\begin{displaymath}
W_{1} ( \Psi^{k}(t) , \Psi^{k}(s)) \leq 2M_{v} (1 + R) |t - s| + \frac{2 \, L_{E, \delta, R}}{m_{1}^{2}} \, |t - s| \,,
\end{displaymath}
where the constants~$L_{E, \delta, R}$ and~$m_{1} = m_{1}(R)$ have been determined in Lemmas~\ref{l:estimate} and~\ref{l:distance}, respectively, and are independent of~$k$, $i$, and~$t$.
\end{proof}

\section*{Acknowledgements}
\noindent The authors wish to thank Giuseppe Savar\'e for a useful discussion regarding the content of Section~\ref{s:alternative}.
The work of MM was partially supported by the \emph{Starting grant per giovani ricercatori} of Politecnico di Torino. The work of FS was supported by the project \emph{Variational methods for stationary and evolution problems with singularities and interfaces} PRIN 2017 financed by the Italian Ministry of Education, University, and Research.
MM and FS are members of the Gruppo Nazionale per l'Analisi Matematica, la Probabilit\`a e le loro Applicazioni (GNAMPA)
of the Instituto Nazionale di Alta Matematica ``Francesco Severi'' (INdAM).
\bibliographystyle{siam}
\bibliography{A_M_S_20.bib}

\end{document}